\newif \ifSISC \SISCfalse

\ifSISC

\documentclass[review,final]{siamart220329} 

\else

\documentclass[a4paper,11pt]{article}
\usepackage[english]{babel}
\usepackage[utf8]{inputenc}
\usepackage[T1]{fontenc}
\usepackage{amsmath,amsfonts,amssymb,amsthm}

\fi

\usepackage{mathtools}
\usepackage{upgreek}
\usepackage{nicefrac}
\usepackage{cancel}
\usepackage{comment}
\usepackage{enumitem}
\usepackage{physics}
\usepackage{newtxtext,newtxmath}
\usepackage{graphicx}

\ifSISC

\else

\usepackage[dvipsnames, svgnames]{xcolor}
\usepackage[hmargin={28mm,28mm},vmargin={30mm,35mm}]{geometry}
\usepackage{authblk}
\usepackage{hyperref}
\hypersetup{
   colorlinks = true,                        
   linkcolor = OliveGreen,                   
   citecolor = violet,                       
   filecolor = blue,                         
   urlcolor = magenta                        
}

\fi

\usepackage{pgfplots}
\pgfplotsset{every axis/.append style={
                     label style={font=\Large},
                     tick label style={font=\Large},
                     legend style={font=\Large}
                     }
}

\ifSISC

\headers{HHO approximations of div-curl systems}{J.~Dalphin, J.-P.~Ducreux, S.~Lemaire, and S.~Pitassi}
\title{Hybrid high-order approximations of div-curl systems on domains with general topology}
\author{J\'er\'emy Dalphin\thanks{ERMES department, EDF R\&D, EDF Lab Paris-Saclay, 91120 Palaiseau, France (\email{jeremy.dalphin@edf.fr}).}
\and Jean-Pierre Ducreux\thanks{ERMES department, EDF R\&D, EDF Lab Paris-Saclay, 91120 Palaiseau, France (\email{jean-pierre.ducreux@edf.fr}).}
\and Simon Lemaire\thanks{Inria, Univ.~Lille, CNRS, UMR 8524 -- Laboratoire Paul Painlev\'e, 59000 Lille, France (\email{simon.lemaire@inria.fr}).}
\and Silvano Pitassi\thanks{Inria, Univ.~Lille, CNRS, UMR 8524 -- Laboratoire Paul Painlev\'e, 59000 Lille \& ERMES department, EDF R\&D, EDF Lab Paris-Saclay, 91120 Palaiseau, France (\email{silvano.pitassi@inria.fr}).}}

\else

\newcommand{\email}[1]{\href{mailto:#1}{#1}}

\title{Hybrid high-order approximations of div-curl systems on domains with general topology}
\author[1]{J\'er\'emy Dalphin\footnote{\email{jeremy.dalphin@edf.fr}}}
\author[1]{Jean-Pierre Ducreux\footnote{\email{jean-pierre.ducreux@edf.fr}}}
\affil[1]{ERMES department, EDF R\&D, EDF Lab Paris-Saclay, 91120 Palaiseau, France}
\author[2]{Simon Lemaire\footnote{\email{simon.lemaire@inria.fr} (corresponding author)}}
\author[2,1]{Silvano Pitassi\footnote{\email{silvano.pitassi@inria.fr}}}
\affil[2]{Inria, Univ.~Lille, CNRS, UMR 8524 -- Laboratoire Paul Painlev\'e, 59000 Lille, France}

\fi

\ifSISC
\else
\newtheorem{theorem}{Theorem}
\newtheorem{lemma}[theorem]{Lemma}

\fi
\newtheorem{remark}[theorem]{Remark}

\newlength{\arrow}
\newcommand*{\myrightarrow}[1]{\xrightarrow{\mathmakebox[\arrow]{#1}}}

\newcommand{\defi}{\mathrel{\mathop:}=}

\newcommand{\st}{~:~}

\newcommand{\di}{{\rm dim}}

\renewcommand{\vec}[1]{\boldsymbol{#1}}
\newcommand{\normal}{\vec{n}}

\newcommand{\Natu}{\mathbb{N}}
\newcommand{\Real}{\mathbb{R}}

\DeclareMathOperator{\Div}{div}
\DeclareMathOperator{\Curl}{\bf curl}
\DeclareMathOperator{\Grad}{\bf grad}

\newcommand{\homol}{\mathfrak{H}}

\newcommand{\Dom}{{\rm Dom}}

\newcommand{\Ker}{{\rm Ker}}
\newcommand{\Ima}{{\rm Im}}

\newcommand{\LL}[1][\Omega]{L^2(#1)}
\newcommand{\bLL}[1][\Omega]{\vec{L}^2(#1)}
\newcommand{\Ho}[1][\Omega]{H^1(#1)}
\newcommand{\bHo}[1][\Omega]{\vec{H}^1(#1)}
\newcommand{\Hzo}{H^1_0(\Omega)}
\newcommand{\HGo}{H^1_{\Gamma}(\Omega)}
\newcommand{\HSo}{H^1_{\Sigma}(\hat{\Omega})}
\newcommand{\Hdiv}[1][\Omega]{\vec{H}(\Div_{\mu} ; #1)}
\newcommand{\Hzdiv}[1][\Omega]{\vec{H}_0(\Div_{\mu} ; #1)}
\newcommand{\Hdivz}[1][\Omega]{\vec{H}(\Div^0_{\mu} ; #1)}
\newcommand{\Hzdivz}[1][\Omega]{\vec{H}_0(\Div^0_{\mu} ; #1)}
\newcommand{\Hcurl}[1][\Omega]{\vec{H}(\Curl ; #1)}
\newcommand{\Hzcurl}[1][\Omega]{\vec{H}_{\vec{0}}(\Curl ; #1)}
\newcommand{\Hcurlz}[1][\Omega]{\vec{H}(\Curl^{\vec{0}} ; #1)}
\newcommand{\Hzcurlz}[1][\Omega]{\vec{H}_{\vec{0}}(\Curl^{\vec{0}} ; #1)}

\newcommand{\fharmonic}{\vec{H}_{\vec{\tau}}(\Omega)}
\newcommand{\sharmonic}{\vec{H}_{n}(\Omega)}

\renewcommand{\d}{\mathcal{D}}
\renewcommand{\t}{\mathcal{T}}
\newcommand{\f}{\mathcal{F}}

\newcommand{\Poly}{\mathcal{P}}
\newcommand{\bPoly}{\vec{\Poly}}

\newcommand{\pGT}[1]{\vec{\mathcal{G}}^{#1}(T)}
\newcommand{\kGT}[1]{\vec{\mathcal{G}}^{{\rm c},#1}(T)}
\newcommand{\pRT}[1]{\vec{\mathcal{R}}^{#1}(T)}

\newcommand{\pRF}[1]{\vec{\mathcal{R}}^{#1}(F)}
\newcommand{\kRF}[1]{\vec{\mathcal{R}}^{{\rm c},#1}(F)}
\newcommand{\pQF}[1]{\vec{\mathcal{Q}}^{#1}(F)}

\newcommand{\bProjT}[1][k]{\vec{\pi}^{#1}_{\bPoly,T}}
\newcommand{\bProjF}[1][k]{\vec{\pi}^{#1}_{\vec{\mathcal{Q}},F}}
\newcommand{\bProjTd}[1][k]{\vec{\pi}^{#1}_{\bPoly,\t}}

\newcommand{\ProjT}[1][k-1]{\pi^{#1}_{\Poly,T}}
\newcommand{\ProjF}[1][k]{\pi^{#1}_{\Poly,F}}
\newcommand{\ProjTd}[1][k-1]{\pi^{#1}_{\Poly,\t}}

\newcommand{\dHcurlT}{\underline{\vec{U}}_T^{k}}
\newcommand{\dHcurl}{\underline{\vec{U}}_{\d}^{k}}
\newcommand{\dHzcurl}{\underline{\vec{U}}_{\d,\vec{0}}^{k}}

\newcommand{\dHoT}{\underline{P}_T^{k}}
\newcommand{\hdHo}{\hat{\underline{P}}_{\d}^{k}}
\newcommand{\hdHSof}{\hat{\underline{P}}_{\d,\Sigma}^{\flat,k}}
\newcommand{\hdHSo}{\hat{\underline{P}}_{\d,\Sigma}^{k}}
\newcommand{\dHo}{\underline{P}_{\d}^{k}}
\newcommand{\dHzo}{\underline{P}_{\d,0}^{k}}
\newcommand{\dHGo}{\underline{P}_{\d,\Gamma}^{k}}

\newcommand{\uud}{\underline{\vec{u}}_{\d}}
\newcommand{\uTd}{\vec{u}_{\t}}
\newcommand{\uuT}{\underline{\vec{u}}_T}

\newcommand{\uhd}{\underline{\vec{h}}_{\d}}
\newcommand{\hTd}{\vec{h}_{\t}}

\newcommand{\uad}{\underline{\vec{a}}_{\d}}
\newcommand{\aTd}{\vec{a}_{\t}}
\newcommand{\uaT}{\underline{\vec{a}}_T}

\newcommand{\uvd}{\underline{\vec{v}}_{\d}}
\newcommand{\vTd}{\vec{v}_{\t}}
\newcommand{\uvT}{\underline{\vec{v}}_T}
\newcommand{\vT}{\vec{v}_T}
\newcommand{\vFt}{\vec{v}_{F,\vec{\tau}}}

\newcommand{\uwd}{\underline{\vec{w}}_{\d}}

\newcommand{\uwT}{\underline{\vec{w}}_T}
\newcommand{\wT}{\vec{w}_T}
\newcommand{\wFt}{\vec{w}_{F,\vec{\tau}}}
\newcommand{\upd}{\underline{p}_{\d}}

\newcommand{\upT}{\underline{p}_T}
\newcommand{\pT}{p_T}

\newcommand{\uqd}{\underline{q}_{\d}}
\newcommand{\qTd}{q_{\t}}
\newcommand{\uqT}{\underline{q}_T}
\newcommand{\qT}{q_T}
\newcommand{\qF}{q_F}

\newcommand{\urd}{\underline{r}_{\d}}

\newcommand{\urT}{\underline{r}_T}
\newcommand{\rT}{r_T}

\newcommand{\intHcurlT}{\underline{\vec{I}}^k_T}
\newcommand{\intHcurl}{\underline{\vec{I}}^k_\d}

\newcommand{\intHoT}{\underline{I}^k_T}
\newcommand{\intHo}{\underline{I}^k_\d}

\newcommand{\CT}{\vec{C}_T^{k-1}}
\newcommand{\CTd}{\vec{C}_{\t}^{k-1}}

\newcommand{\GT}{\vec{G}_T^{k}}
\newcommand{\GTd}{\vec{G}_{\t}^{k}}

\newcommand{\logLogSlopeTriangle}[5]
{
    \pgfplotsextra
    {
        \pgfkeysgetvalue{/pgfplots/xmin}{\xmin}
        \pgfkeysgetvalue{/pgfplots/xmax}{\xmax}
        \pgfkeysgetvalue{/pgfplots/ymin}{\ymin}
        \pgfkeysgetvalue{/pgfplots/ymax}{\ymax}

        \pgfmathsetmacro{\xArel}{#1}
        \pgfmathsetmacro{\yArel}{#3}
        \pgfmathsetmacro{\xBrel}{#1-#2}
        \pgfmathsetmacro{\yBrel}{\yArel}
        \pgfmathsetmacro{\xCrel}{\xArel}

        \pgfmathsetmacro{\lnxB}{\xmin*(1-(#1-#2))+\xmax*(#1-#2)} 
        \pgfmathsetmacro{\lnxA}{\xmin*(1-#1)+\xmax*#1} 
        \pgfmathsetmacro{\lnyA}{\ymin*(1-#3)+\ymax*#3} 
        \pgfmathsetmacro{\lnyC}{\lnyA+#4*(\lnxA-\lnxB)}
        \pgfmathsetmacro{\yCrel}{\lnyC-\ymin)/(\ymax-\ymin)}

        \coordinate (A) at (rel axis cs:\xArel,\yArel);
        \coordinate (B) at (rel axis cs:\xBrel,\yBrel);
        \coordinate (C) at (rel axis cs:\xCrel,\yCrel);

        \draw[#5]   (A)-- node[pos=0.5,anchor=north] {\scriptsize{1}}
                    (B)-- 
                    (C)-- node[pos=0.,anchor=west] {\scriptsize{#4}} 
                    (A);
    }
}

\newcommand{\logLogSlopeTriangleNDOFs}[5]
{
    \pgfplotsextra
    {
        \pgfkeysgetvalue{/pgfplots/xmin}{\xmin}
        \pgfkeysgetvalue{/pgfplots/xmax}{\xmax}
        \pgfkeysgetvalue{/pgfplots/ymin}{\ymin}
        \pgfkeysgetvalue{/pgfplots/ymax}{\ymax}

        \pgfmathsetmacro{\xArel}{#1}
        \pgfmathsetmacro{\yArel}{#3}
        \pgfmathsetmacro{\xBrel}{#1-#2}
        \pgfmathsetmacro{\yBrel}{\yArel}
        \pgfmathsetmacro{\xCrel}{\xArel}

        \pgfmathsetmacro{\lnxB}{\xmin*(1-(#1-#2))+\xmax*(#1-#2)} 
        \pgfmathsetmacro{\lnxA}{\xmin*(1-#1)+\xmax*#1} 
        \pgfmathsetmacro{\lnyA}{\ymin*(1-#3)+\ymax*#3} 
        \pgfmathsetmacro{\lnyC}{\lnyA-#4*(\lnxA-\lnxB)}
        \pgfmathsetmacro{\yCrel}{\lnyC-\ymin)/(\ymax-\ymin)}

        \coordinate (A) at (rel axis cs:\xArel,\yArel);
        \coordinate (B) at (rel axis cs:\xBrel,\yBrel);
        \coordinate (C) at (rel axis cs:\xCrel,\yCrel);

        \draw[#5]   (A)-- node[pos=0.5,anchor=north] {\scriptsize{1}}
                    (B)-- 
                    (C)-- node[pos=0.,anchor=east] {\scriptsize{#4}} 
                    (A);
    }
}

\begin{document}

\maketitle

\ifSISC

\begin{abstract}
  We devise and analyze hybrid polyhedral methods of arbitrary order for the approximation of div-curl systems on three-dimensional domains featuring non-trivial topology. The div-curl systems we are interested in stem from magnetostatics, and can either be first-order (field formulation) or second-order (vector potential formulation). The well-posedness of the resulting discrete problems essentially hinges on recently established, topologically generic, hybrid versions of the (first and second) Weber inequalities. Our error analysis covers the case of regular solutions. Leveraging (co)homology computation techniques from the literature, we perform an in-depth numerical assessment of our approach, covering, in particular, the case of non-simply-connected domains.
\end{abstract}

\else

\begin{abstract}
  We devise and analyze hybrid polyhedral methods of arbitrary order for the approximation of div-curl systems on three-dimensional domains featuring non-trivial topology. The div-curl systems we are interested in stem from magnetostatics, and can either be first-order (field formulation) or second-order (vector potential formulation). The well-posedness of the resulting discrete problems essentially hinges on recently established, topologically generic, hybrid versions of the (first and second) Weber inequalities. Our error analysis covers the case of regular solutions. Leveraging (co)homology computation techniques from the literature, we perform an in-depth numerical assessment of our approach, covering, in particular, the case of non-simply-connected domains.
\end{abstract}

\fi

\section{Introduction} \label{se:intro}

Let $\Omega$ be a domain in $\Real^3$, i.e.~a bounded and connected Lipschitz open set of $\Real^3$.
In the $L^2$ framework, the standard (primal) de Rham complex reads:
\begin{equation} \label{eq:derham}
  \{0\}\myrightarrow{0}\Ho\myrightarrow{\Grad}\Hcurl\myrightarrow{\Curl}\vec{H}(\Div;\Omega)\myrightarrow{\Div}\LL\myrightarrow{0}\{0\},
\end{equation}
with cohomology spaces
\begin{equation} \label{eq:cohom}
  \begin{aligned}
    \homol^0&\defi\Ker(\Grad)(/\Ima(0)),\qquad&\homol^1&\defi\Ker(\Curl)/\Ima(\Grad),\\
    \homol^2&\defi\Ker(\Div)/\Ima(\Curl),\qquad&\homol^3&\defi(\Ker(0)=)\LL/\Ima(\Div).
  \end{aligned}
\end{equation}
The spaces $\homol^n$ have respective dimensions equal to the Betti numbers $\beta_n$, with $\beta_0=1$ (number of connected components of $\Omega$), $\beta_1\in\Natu$ (number of tunnels crossing through $\Omega$), $\beta_2\in\Natu$ (number of voids encapsulated by $\Omega$), and $\beta_3=0$.
In turn, the dual de Rham complex reads:
\begin{equation} \label{eq:derham.adj}
  \{0\}\myrightarrow{0}\Hzo\myrightarrow{\Grad_{\,0}}\Hzcurl\myrightarrow{\Curl_{\,\vec{0}}}\vec{H}_0(\Div;\Omega)\myrightarrow{\Div_{\,0}}\LL\myrightarrow{0}\{0\},
\end{equation}
where the operators $\Grad_{\,0}$, $\Curl_{\,\vec{0}}$, and $\Div_{\,0}$ are respectively $L^2$-adjoint to the operators $-\Div$, $\Curl$, and $-\Grad$.
The zero subscripts aim to recall that the domains of the adjoint differential operators embed zero boundary conditions.
The corresponding homology spaces are
\begin{equation} \label{eq:cohom.adj}
  \begin{aligned}
    \homol^0_0&\defi\Ker(\Grad_{\,0})(/\Ima(0)),\qquad&\homol^1_0&\defi\Ker(\Curl_{\,\vec{0}})/\Ima(\Grad_{\,0}),\\
    \homol^2_0&\defi\Ker(\Div_{\,0})/\Ima(\Curl_{\,\vec{0}}),\qquad&\homol^3_0&\defi(\Ker(0)=)\LL/\Ima(\Div_{\,0}).
  \end{aligned}
\end{equation}
Importantly, there holds $\homol^n_0=\homol^{3-n}$ (this is the so-called {\em Poincar\'e--Lefschetz duality}).

Let $\mu$ be a possibly varying coefficient, assumed to be uniformly bounded by above and by below by some positive constants. In what follows, the symbols $\perp$ and $\perp_\mu$ respectively denote standard and $\mu$-weighted $L^2$-orthogonality. Let the sequence of differential operators $\{\texttt{Grad},\texttt{Curl},\texttt{Div}\}$ stand either (i) for $\{\Grad,\Curl,\Div\}$ from the primal de Rham complex~\eqref{eq:derham}, or (ii) for $\{\Grad_{\,0},\Curl_{\,\vec{0}},\Div_{\,0}\}$ from the dual de Rham complex~\eqref{eq:derham.adj}. In the first case, let $\mathfrak{h}^n\defi\homol^n$; in the second, let $\mathfrak{h}^n\defi\homol^n_0=\homol^{3-n}$. We aim to study problems of the following form. Given some datum $\vec{j}\in\Ker(\texttt{Div})$, find $\vec{h}\in\Dom(\texttt{Curl})\cap\Ima(\texttt{Grad})^{\perp_\mu}$ s.t.
\begin{equation} \label{eq:prob.fo}
  \texttt{Curl}(\vec{h})=\vec{j}.
\end{equation}
One can easily verify that $\texttt{Curl}$, seen as a bounded linear mapping from $\Dom(\texttt{Curl})\cap\Ima(\texttt{Grad})^{\perp_\mu}$ to $\Ker(\texttt{Div})$, is a Fredholm operator, with null space $\Ker(\texttt{Curl})\cap\Ima(\texttt{Grad})^{\perp_\mu}\cong\mathfrak{h}^1$, and defect space $\Ker(\texttt{Div})/\Ima(\texttt{Curl})=\mathfrak{h}^2$. Thus, the following alternative holds true for Problem~\eqref{eq:prob.fo}:
\begin{itemize}
  \item[$\bullet$] either $\di(\mathfrak{h}^1)=\di(\mathfrak{h}^2)=0$, then Problem~\eqref{eq:prob.fo} admits a unique solution;
  \item[$\bullet$] in the opposite case, for Problem~\eqref{eq:prob.fo} to admit a solution, it is necessary that $\vec{j}\in\Ima(\texttt{Curl})$, which amounts, for $\vec{j}\in\Ker(\texttt{Div})$, to additionally satisfying $\vec{j}\perp\mathfrak{h}^2$; the uniqueness of the solution is then recovered by imposing $\di(\mathfrak{h}^1)$ additional constraints to $\vec{h}$ (which already satisfies $\vec{h}\perp_{\mu}\Ima(\texttt{Grad})$).
\end{itemize}
Problem~\eqref{eq:prob.fo} is prototypical of the first-order formulation (also known as {\it field formulation}) of magnetostatics.
In practice, $\vec{j}:\Omega\to\Real^3$ is a given electric current density, $\mu:\Omega\to\Real_{>0}$ is the magnetic permeability of the medium, and $\vec{h}:\Omega\to\Real^3$ is the sought magnetic field. 
The model can be endowed with either normal or tangential boundary conditions, depending on which version (respectively,~\eqref{eq:derham} or~\eqref{eq:derham.adj}) of the de Rham complex the sequence of differential operators stems from. We refer the reader to Section~\ref{ssse:field} for a somewhat more conventional rewriting of Problem~\eqref{eq:prob.fo} in the case of normal boundary conditions.

Let us now introduce $\texttt{Curl}^\star$, $L^2$-adjoint of $\texttt{Curl}$. Recall that $\Ker(\texttt{Curl})^\perp=\Ima(\texttt{Curl}^\star)$. If the solution $\vec{h}\perp_{\mu}\Ima(\texttt{Grad})$ to Problem~\eqref{eq:prob.fo} (assuming it exists) additionally satisfies $\vec{h}\perp_{\mu}\mathfrak{h}^1$, then it can be written $\vec{h}=\mu^{-1}\texttt{Curl}^\star(\vec{a})$ for some vector potential $\vec{a}\in\Dom(\texttt{Curl}^\star)$. Obviously, the vector potential $\vec{a}$ is non-unique; it is only defined up to an element of $\Ker(\texttt{Curl}^\star)$ (gauge). Gauging out to zero, Problem~\eqref{eq:prob.fo} equivalently rewrites: find $\vec{a}\in\Dom(\texttt{Curl}^\star)\cap\Ker(\texttt{Curl}^\star)^\perp$ s.t.
\begin{equation} \label{eq:prob.so}
  \texttt{Curl}\big(\mu^{-1}\texttt{Curl}^\star(\vec{a})\big)=\vec{j},
\end{equation}
and the (unique) solution to this problem does exist as soon as $\vec{j}\in\Ker(\texttt{Div})$ additionally satisfies $\vec{j}\perp\mathfrak{h}^2$.
Problem~\eqref{eq:prob.so} is prototypical of the second-order formulation (also known as {\it vector potential formulation}) of magnetostatics. In practice, $\vec{a}:\Omega\to\Real^3$ is the magnetic vector potential.
When Problem~\eqref{eq:prob.fo} is endowed with normal (resp.~tangential) boundary conditions, Problem~\eqref{eq:prob.so} is endowed with tangential (resp.~normal) boundary conditions. In the second scenario, part of the boundary conditions in Problem~\eqref{eq:prob.so} are actually natural, and encode the tangential prescription on $\vec{h}$.
For a more conventional rewriting of Problem~\eqref{eq:prob.so} in the case of tangential boundary conditions (first scenario), we refer to Section~\ref{ssse:vecpot}.

In this work, we aim at devising arbitrary-order {\em hybrid} polyhedral discretizations of the first- and second-order model problems~\eqref{eq:prob.fo} and~\eqref{eq:prob.so}. By definition, hybrid methods only attach unknowns to the faces and to the cells of the spatial partition at hand. Examples of such approaches are the Hybridizable Discontinuous Galerkin (HDG)~\cite{CoGoL:09}, the Weak Galerkin (WG)~\cite{WanYe:13}, the Hybrid High-Order (HHO)~\cite{DPErn:15,DPELe:14}, or the non-conforming Virtual Element (ncVE)~\cite{LipMa:14,AdDLM:16} methods.
As might be expected, and as was first documented in~\cite{Cockb:16,CDPEr:16}, all these technologies are very tightly connected.
The devising of hybrid polyhedral methods for div-curl systems has already been addressed to some extent in the HDG and WG literatures. For second-order models, these contributions include~\cite{NPCoc:11,CQSSo:17,CCuXu:19} (cf.~also~\cite{DuSay:20}) for HDG, and~\cite{MWYZh:15} for WG, all of them restricted to trivial topologies. Regarding now first-order models, the only contributions we are aware of are~\cite{LiYeZ:18}, in which a primal WG method is introduced, and~\cite{WangW:16,CaoWW:22,CaoWW:23}, in which alternative (primal-dual) WG approaches are studied. In all these contributions, full face polynomial spaces are employed to discretize the vectorial variable. Also, either the case of non-trivial topologies is not covered (like in~\cite{LiYeZ:18}), or it is treated but in a non-robust fashion. Recently, the discretization of both first- and second-order models of magnetostatics has been studied in the HHO context~\cite{CDPLe:22}. Therein, the vectorial variable is sought into trimmed face polynomial spaces, reusing and extending ideas introduced in~\cite{CQSSo:17} (cf.~also~\cite{LeSch:16}). The contribution~\cite{CDPLe:22}, however, solely covers the case of trivial topologies. Our objective in the present work is to fill this gap. We aim at devising, analyzing, and numerically assessing trimmed HHO methods for (first- and second-order) models set in domains with arbitrary topology.
With respect to~\cite{WangW:16,CaoWW:22,CaoWW:23}, apart from using trimmed face spaces, our approach for first-order models is also based on a different, more suitable variational formulation (inspired from~\cite{Kikuc:89}), which allows for both a leaner construction and a robust handling of harmonic fields. In turn, our stability analysis hinges on the systematic use of the (topologically generic) hybrid Weber inequalities recently established in~\cite{LePit:25} by the last two authors. Finally, our error analysis (cf.~Theorems~\ref{th:field.esti} and~\ref{th:vecpot.esti}) covers the case of regular solutions.
Let us emphasize that~\cite{LePit:25} is exclusively concerned with general-purpose discrete functional analysis; it does not address the practical devising of hybrid methods, which is the subject of the present work.
To the best of our knowledge, in the realm of polyhedral approaches, our work is the first to deal, from the devising to the implementation in 3D, and in a systematic and robust way, with non-trivial topologies.
Let us stress that, from an application point of view, non-trivial topologies are ubiquitous (think, e.g., of the toroidal vacuum chamber of a tokamak). Their proper handling is thus of paramount importance.
From an algorithmic point of view, arbitrary topologies come with a number of additional difficulties. It is especially true for non-simply-connected domains, for which discrete cutting surfaces might need to be computed from the polyhedral mesh at hand (see Figure~\ref{fig:cut}).

The material is organized as follows. In Section~\ref{se:prelim}, we introduce both the topological and functional frameworks, and we recall the first and second Helmholtz--Hodge decompositions. In Section~\ref{se:dis.set}, we introduce the discrete setting, including polyhedral discretizations, polynomial decompositions, and hybrid spaces. In Section~\ref{se:hho}, we devise and analyze HHO methods for (first- and second-order) model problems of the forms~\eqref{eq:prob.fo} and~\eqref{eq:prob.so}. Finally, in Section~\ref{se:num}, we provide a comprehensive set of numerical experiments on non-trivial 3D domains, assessing the relevance of our methodology.

\section{Preliminaries} \label{se:prelim}

\subsection{Topological framework} \label{sse:top}

We recall that $\Omega$ denotes a domain in $\Real^3$, that is a bounded and connected Lipschitz open set of $\Real^3$. Let $\Gamma\defi\partial\Omega$ denote its boundary, and $\normal:\Gamma\to\Real^3$ be the (almost everywhere defined) unit vector field normal to $\Gamma$, pointing outward from $\Omega$. We recall that the Betti numbers $\beta_1$ and $\beta_2$ of $\Omega$ respectively count the number of tunnels crossing through $\Omega$ ($\beta_1\in\Natu$), and the number of voids encapsulated by $\Omega$ ($\beta_2\in\Natu$). For $\Omega$ simply-connected, $\beta_1=0$. Likewise, when the boundary $\Gamma$ of $\Omega$ is connected, $\beta_2=0$. In what follows, when both $\beta_1,\beta_2$ are equal to zero, we say that the topology of $\Omega$ is trivial.

Whenever $\beta_1>0$, we make the following classical assumption: there exist $\beta_1$ 
orientable and connected two-dimensional manifolds with boundary, denoted $\Sigma_1,\ldots,\Sigma_{\beta_1}$, called {\em cutting surfaces}, satisfying $\Sigma_i\subset\Omega$ and $\partial\Sigma_i\subset\Gamma$ for all $i\in\{1,\ldots,\beta_1\}$, such that the open set $\hat{\Omega}\defi\Omega\setminus\cup_{i\in\{1,\ldots,\beta_1\}}\Sigma_i$ is simply-connected, i.e.~its first Betti number is zero (see Figure~\ref{fig:hollow.torus.geo} for an example). 
We will assume, in what follows, that $\hat{\Omega}$ is connected (which is generally the case starting from a connected domain $\Omega$), and that the cutting surfaces are sufficiently regular so that the set $\hat{\Omega}$ is pseudo-Lipschitz (cf.~\cite[Def.~3.2.2]{ACJLa:18}). For any $i\in\{1,\ldots,\beta_1\}$, we let $\normal_{\Sigma_i}:\Sigma_i\to\Real^3$ be the (almost everywhere defined) unit vector field normal to $\Sigma_i$, the orientation of which we arbitrarily prescribe (the cutting surfaces $\Sigma_i$ are supposed to be orientable). The orientation of $\normal_{\Sigma_i}$ being prescribed, we associate the tag ``$+$'' (resp.~``$-$'') to the side of $\hat{\Omega}$ (with respect to $\Sigma_i$) for which $\normal_{\Sigma_i}$ points outward (resp.~inward).
\begin{remark}[Cutting surfaces]
  In~\cite{BFGhi:12}, the class of (so-called) weakly-Helmholtz domains is introduced. Its definition does not assume that the cut domain $\hat{\Omega}$ is simply-connected, which allows to cover more exotic topologies (e.g., the domain obtained as the complement in a cube of a trefoil knot). Our analysis could be extended to cover weakly-Helmholtz domains. However, whereas weakly-Helmholtz domains assume non-intersecting cutting surfaces, we do not make such an assumption here. This allows us to consider the non-weakly-Helmholtz, yet physically relevant configuration of a toroidal vacuum chamber (see Figure~\ref{fig:hollow.torus.geo}), for which our approach seamlessly applies, as supported by the numerical results in Section~\ref{sse:hollow.torus}.
\end{remark}

Whenever $\beta_2>0$, letting $\Gamma_0$ be the (connected) boundary of the only unbounded component of the exterior open set $\Real^3\setminus\overline{\Omega}$, there exist $\beta_2$ (maximally) connected components $\Gamma_1,\ldots,\Gamma_{\beta_2}$ of $\Gamma$ such that $\Gamma=\cup_{j\in\{0,\ldots,\beta_2\}}\Gamma_j$ (cf.~Figure~\ref{fig:hollow.torus.geo}). If $\beta_2=0$, there holds $\Gamma=\Gamma_0$.

\subsection{Functional framework}

Let $\mu:\Omega\to\Real$ be a given function satisfying, for real numbers $0<\mu_\flat\leq\mu_\sharp<\infty$,
\begin{equation} \label{eq:mu.c}
  \mu_\flat\leq\mu(\vec{x})\leq\mu_\sharp\qquad\text{for a.e.}~\vec{x}\in\Omega.
\end{equation}

For $m\in\{2,3\}$, and for $X$ an $m$-dimensional, (relatively) open pseudo-Lipschitz subset of $\overline{\Omega}$, we let $\LL[X]$ (respectively, $\bLL[X]$) denote the Lebesgue space of square-integrable functions (respectively, $\Real^m$-valued vector fields) over $X$. The standard inner products (and norms) in $\LL[X]$ and $\bLL[X]$ are irrespectively denoted by $(\mathfrak{f},\mathfrak{g})_X\defi\int_X\mathfrak{f}{\cdot}\mathfrak{g}$ (and $\|{\cdot}\|_{0,X}\defi\sqrt{({\cdot},{\cdot})_X}$).
We also define $L^2_0(X)\defi\left\{v\in\LL[X]\mid\int_X v=0\right\}$ and $\vec{L}^2_{\vec{0}}(X)\defi\left\{\vec{v}\in\bLL[X]\mid\int_X\vec{v}=\vec{0}\right\}$.
For $s>0$, we let $H^s(X)$ (resp.~$\vec{H}^s(X)$) denote the Sobolev space of functions in $\LL[X]$ (resp.~$\Real^m$-valued vector fields in $\bLL[X]$) possessing square-integrable partial weak derivatives up to order $s$ over $X$ (for fractional $s$, we follow the classical Sobolev--Slobodeckij construction). The standard norms (and semi-norms) in $H^s(X)$ and $\vec{H}^s(X)$ are irrespectively denoted by $\|{\cdot}\|_{s,X}$ (and $|{\cdot}|_{s,X}$).
Given a disjoint Lipschitz partition $\mathcal{K}$ of $\Omega$ (in the sense that every $K\in\mathcal{K}$ is an open Lipschitz set, $K_1\cap K_2=\emptyset$ for all $K_1,K_2\in\mathcal{K}$ with $K_1\neq K_2$, and $\bigcup_{K\in\mathcal{K}}\overline{K}=\overline{\Omega}$), we also define the broken Sobolev space $\vec{H}^s(\mathcal{K})\defi\left\{\vec{v}\in\bLL[\Omega]\st\vec{v}_{\mid K}\in\vec{H}^s(K)\;\forall K\in\mathcal{K}\right\}$.

Let $Y$ be a three-dimensional, open Lipschitz subset of $\Omega$. Classically, we let
\begin{align*}
  \Hcurl[Y]&\defi\big\{\vec{v}\in\bLL[Y]\mid\Curl\vec{v}\in\bLL[Y]\big\},\\
  \Hdiv[Y]&\defi\big\{\vec{v}\in\bLL[Y]\mid\Div(\mu\vec{v})\in\LL[Y]\big\},
\end{align*}
with $\vec{H}(\Div;Y)\defi\vec{H}(\Div_1;Y)$, as well as their two subspaces $\Hcurlz[Y]\defi\big\{\vec{v}\in\Hcurl[Y]\mid\Curl\vec{v}\equiv\vec{0}\big\}$ and $\Hdivz[Y]\defi\big\{\vec{v}\in\Hdiv[Y]\mid\Div(\mu\vec{v})\equiv 0\big\}$.
Let $\normal_{\partial Y}:\partial Y\to\Real^3$ denote the (almost everywhere defined) outward unit normal vector to $\partial Y$.
For $\vec{v}\in\Hdiv[Y]$, the normal trace of $\mu\vec{v}$ on $\partial Y$ can be defined as an element of $H^{-\frac12}(\partial Y)$ (space of bounded linear forms on $H^{\frac12}(\partial Y)$), denoted $(\mu\vec{v})_{\mid\partial Y}{\cdot}\normal_{\partial Y}$.
Likewise, for $\vec{v}\in\Hcurl[Y]$, one can give a sense to the rotated tangential trace of $\vec{v}$ on $\partial Y$ as an element of $H^{-\frac12}(\partial Y)^3$ (space of bounded linear forms on $H^{\frac12}(\partial Y)^3$), denoted $\vec{v}_{\mid\partial Y}{\times}\normal_{\partial Y}$.
Following~\cite[Rmk.~1]{LePit:25}, from now on, tangential vector fields are identified to two-dimensional vector fields. The relevant boundary functional spaces then become $\bLL[\partial Y]=\LL[\partial Y]^2$, $\vec{H}^{\frac12}(\partial Y)=H^{\frac12}(\partial Y)^2$, and $\vec{H}^{-\frac12}(\partial Y)\defi H^{-\frac12}(\partial Y)^2$, and one may abuse the notation and write $\vec{v}_{\mid\partial Y}{\times}\normal_{\partial Y}\in\vec{H}^{-\frac12}(\partial Y)$.
In what follows, the duality pairings between $H^{-\frac12}(\partial Y)$ and $H^{\frac12}(\partial Y)$ on the one side, and between $\vec{H}^{-\frac12}(\partial Y)$ and $\vec{H}^{\frac12}(\partial Y)$ on the other side, are irrespectively denoted by $\langle{\cdot},{\cdot}\rangle_{\partial Y}$.
If $\vec{v}\in\Hcurl[Y]\cap\vec{H}(\Div;Y)\cap\vec{H}^s(Y)$ for some $s>\frac{1}{2}$, then for a.e.~$\vec{x}\in\partial Y$, $\vec{v}_{\mid\partial Y}(\vec{x})=(\vec{v}_{\mid\partial Y}{\cdot}\normal_{\partial Y})(\vec{x})\normal_{\partial Y}(\vec{x})+\normal_{\partial Y}(\vec{x}){\times}(\vec{v}_{\mid\partial Y}{\times}\normal_{\partial Y})(\vec{x})$.
In this case, $\vec{v}_{\mid\partial Y}{\cdot}\normal_{\partial Y}\in\LL[\partial Y]$, and $\vec{v}_{\mid\partial Y}{\times}\normal_{\partial Y}\in\bLL[\partial Y]$.

\ifSISC
To account for essential boundary conditions, we shall define the subspaces $\Hzcurl\defi\left\{\vec{v}\in\Hcurl\mid\vec{v}_{\mid\Gamma}{\times}\normal\equiv\vec{0}\right\}$ (along with $\Hzcurlz\defi\Hzcurl\cap\Hcurlz$), and $\Hzdiv\defi\left\{\vec{v}\in\Hdiv\mid(\mu\vec{v})_{\mid\Gamma}{\cdot}\normal\equiv 0\right\}$ (along with $\Hzdivz\defi\Hzdiv\cap\Hdivz$).
\else
To account for essential boundary conditions, we shall define the subspaces
\begin{equation*}
  \Hzcurl\defi\left\{\vec{v}\in\Hcurl\mid\vec{v}_{\mid\Gamma}{\times}\normal\equiv\vec{0}\right\}
\end{equation*}
(along with $\Hzcurlz\defi\Hzcurl\cap\Hcurlz$), and
\begin{equation*}
  \Hzdiv\defi\left\{\vec{v}\in\Hdiv\mid(\mu\vec{v})_{\mid\Gamma}{\cdot}\normal\equiv 0\right\}
\end{equation*}
(with $\Hzdivz\defi\Hzdiv\cap\Hdivz$).
\fi
We also set $\Hzo\defi\{v\in\Ho\!\mid\! v_{\mid\Gamma}\equiv 0\}$.
Last, assume that the first Betti number $\beta_1$ of $\Omega$ is positive. Then, for $v\in\LL[\hat{\Omega}]$ (resp.~for $\vec{v}\in\bLL[\hat{\Omega}]$), we denote by $\check{v}$ (resp.~by $\check{\vec{v}}$) its continuation to $\LL$ (resp.~to $\bLL$). Also, for any $v:\hat{\Omega}\to\Real$, and $i\in\{1,\ldots,\beta_1\}$, denoting $v^+_{\mid\Sigma_i}$ and $v^-_{\mid\Sigma_i}$ the traces of $v$ on $\Sigma_i$ defined (if need be, in a weak sense) from both sides of $\hat{\Omega}$ (respectively tagged by ``$+$'' and ``$-$''), we define the jump of $v$ across $\Sigma_i$ by
\begin{equation} \label{eq:jump}
  \llbracket v\rrbracket_{\Sigma_i}\defi v^+_{\mid\Sigma_i}-v^-_{\mid\Sigma_i}.
\end{equation}                         

\subsection{Helmholtz--Hodge decompositions}

We collect classical results about Helmholtz--Hodge decompositions; for further details, we refer the reader to~\cite{GiRav:86,DaLio:90,ABDGi:98} (cf.~also~\cite{GroKo:04,ACJLa:18}, as well as~\cite[Sec.~2.3]{LePit:25}).
Before proceeding, we define $\vec{L}^2_\mu(\Omega)\defi\left(\bLL,(\mu\,\cdot,{\cdot})_{\Omega}\right)$.

\subsubsection{1$^{\text{st}}$ Helmholtz--Hodge decomposition}

Let us first consider the harmonic space $\fharmonic\defi\Hzcurlz\cap\Hdivz$ which, by~\eqref{eq:cohom.adj}, satisfies $\fharmonic\cong\homol^1_0$. In particular, the space $\fharmonic$ has dimension $\beta_2$, and it can be proved that vector fields $\vec{w}\in\fharmonic$ are entirely characterized by the data of $\big(\langle(\mu\vec{w})_{\mid\Gamma_{j}}{\cdot}\normal,1\rangle_{\Gamma_{j}}\in\Real\big)_{j\in\{1,\ldots,\beta_2\}}$, where $\langle\cdot,\cdot\rangle_{\Gamma_j}$ here stands for the duality pairing in $H^{\frac12}(\Gamma_j)$.
The following $\vec{L}^2_\mu(\Omega)$-orthogonal Helmholtz--Hodge decomposition holds true:
\begin{equation} \label{eq:helm1}
  \bLL=\Grad\big(\Hzo\big)\overset{\perp_\mu}{\oplus}\frac{\mu_\sharp}{\mu}\Curl\big(\bHo\cap\vec{L}^2_{\vec{0}}(\Omega)\big)\overset{\perp_\mu}{\oplus}\fharmonic.
\end{equation}
Furthermore, letting
\begin{equation} \label{eq:HGo}
  \HGo\defi\big\{v\in\Ho,\,v_{\mid\Gamma_0}\equiv 0\mid\exists\,(\gamma_j)\in\Real^{\beta_2}, v_{\mid\Gamma_j}\equiv\gamma_j\,\forall j\in\{1,\ldots,\beta_2\}\big\},
\end{equation}
it can be noticed that
\begin{equation} \label{eq:GradHGo}
  \Grad\big(\HGo\big)=\Grad\big(\Hzo\big)\overset{\perp_\mu}{\oplus}\fharmonic.
\end{equation}
Remark, also, that $\Grad\big(\HGo\big)\subset\Hzcurlz$.

\subsubsection{2$^{\text{nd}}$ Helmholtz--Hodge decomposition}

Let us now consider the harmonic space $\sharmonic\defi\Hcurlz\cap\Hzdivz$ which, by~\eqref{eq:cohom.adj}, satisfies $\sharmonic\cong\homol^2_0$. In particular, the space $\sharmonic$ has dimension $\beta_1$, and it can be proved that vector fields $\vec{w}\in\sharmonic$ are entirely characterized by the data of $\big(\langle(\mu\vec{w})_{\mid\Sigma_{i}}{\cdot}\normal_{\Sigma_{i}},1\rangle_{\Sigma_{i}}\in\Real\big)_{i\in\{1,\ldots,\beta_1\}}$, where $\langle\cdot,\cdot\rangle_{\Sigma_i}$ here stands for the duality pairing in $H^{\frac12}(\Sigma_i)$.
The following $\vec{L}^2_\mu(\Omega)$-orthogonal Helmholtz--Hodge decomposition holds true:
\begin{equation} \label{eq:helm2}
  \bLL=\Grad\big(\Ho\cap L^2_0(\Omega)\big)\overset{\perp_\mu}{\oplus}\frac{\mu_\sharp}{\mu}\Curl\big(\bHo\cap\Hzcurl\big)\overset{\perp_\mu}{\oplus}\sharmonic.
\end{equation}
Furthermore, letting
\begin{equation} \label{eq:HSo}
  \HSo\defi\big\{v\in\Ho[\hat{\Omega}]\cap L^2_0(\hat{\Omega})\mid\exists\,(\sigma_i)\in\Real^{\beta_1}, \llbracket v\rrbracket_{\Sigma_i}\equiv\sigma_i\,\forall i\in\{1,\ldots,\beta_1\}\big\},
\end{equation}
it can be noticed that
\begin{equation} \label{eq:GradHSo}
  \check{\Grad}\big(\HSo\big)=\Grad\big(\Ho\cap L^2_0(\Omega)\big)\overset{\perp_\mu}{\oplus}\sharmonic.
\end{equation}
We remind the reader that, for $v\in\Ho[\hat{\Omega}]$, $\check{\Grad}\,v$ is the continuation to $\bLL$ of $\Grad v\in\bLL[\hat{\Omega}]$, and that the jump $\llbracket{\cdot}\rrbracket_{\Sigma_{i}}$ is defined in~\eqref{eq:jump}.
Note also that $\check{\Grad}\big(\HSo\big)\subset\Hcurlz$.

\section{Discrete setting} \label{se:dis.set}

From now on, we assume that the domain $\Omega\subset\Real^3$ is a (Lipschitz) polyhedron.

\subsection{Polyhedral discretizations} \label{sse:pol.dis}

We consider discretizations $\d\defi(\t,\f)$ of $\Omega\subset\Real^3$ in the sense of~\cite[Def.~1.4]{DPDro:20}.

The set $\t$ is a finite collection of disjoint open Lipschitz polyhedra $T$ (the mesh cells), which is assumed to form a partition of the domain, that is $\overline{\Omega} = \bigcup_{T\in\t} \overline{T}$.
For all $T\in\t$, we let $h_T\defi\max_{\vec{x},\vec{y}\in\overline{T}}|\vec{x}-\vec{y}|$ denote the diameter of the cell $T$. We also let $h_\t$ be s.t.~$h_{\t\mid T}\defi h_T$ for all $T\in\t$, and we define the mesh size by $h_\d \defi \max_{T\in\t} h_T$.
In turn, the set $\f$ is a finite collection of disjoint connected subsets of $\overline{\Omega}$ (the mesh faces) such that, for all $F\in\f$,
\begin{enumerate}
  \item[(i)] $F$ is a relatively open, Lipschitz polygonal subset of an affine hyperplane, and
  \item[(ii)] either there are two distinct mesh cells $T^+,T^-\in\t$ s.t.~$\overline{F}\subseteq\partial T^+\cap\partial T^-$ ($F$ is then an interface), or there is a mesh cell $T\in\t$ s.t.~$\overline{F}\subseteq\partial T\cap\Gamma$ ($F$ is then a boundary face).
\end{enumerate}
The set of mesh faces is assumed to form a partition of the mesh skeleton, that is to satisfy $\bigcup_{T\in\t}\partial T=\bigcup_{F\in\f}\overline{F}$.
For all $F\in\f$, we let $h_F\defi\max_{\vec{x},\vec{y}\in\overline{F}}|\vec{x}-\vec{y}|$ denote the diameter of the face $F$.
Interfaces are collected in the set $\f^\circ$, whereas boundary faces are collected in the set $\f^\partial$.
For all $T\in\t$, we denote by $\f_T$ the subset of $\f$ which collects the mesh faces lying on the boundary of $T$, so that $\partial T=\bigcup_{F\in\f_T}\overline{F}$.
For all $T\in\t$, consistently with our notation so far, we let $\normal_{\partial T}:\partial T\to\Real^3$ denote the (almost everywhere defined) unit vector field normal to $\partial T$, pointing outward from $T$. For all $F\in\f_T$, we also let $\normal_{T,F}\defi\normal_{\partial T\mid F}$ be the (constant) unit vector normal to the hyperplane containing $F$, and pointing outward from $T$.
For all $F\in\f$, we define $\normal_F$ as the (constant) unit vector normal to $F$ such that either $\normal_F\defi\normal_{T^+,F}$ if $F\subset\partial T^+\cap\partial T^-\in\f^\circ$, or $\normal_F\defi\normal_{T,F}(=\normal_{\mid F})$ if $F\subset\partial T\cap\Gamma\in\f^\partial$. For further use, we also let, for all $T\in\t$ and all $F\in\f_T$, $\varepsilon_{T,F}\in\{-1,1\}$ be such that $\varepsilon_{T,F}\defi\normal_{T,F}{\cdot}\normal_F$.
Finally, for any $X\in\f\cup\t$, we let $\vec{x}_X\in\Real^3$ be some point inside $X$.

Whenever $\beta_1>0$, for all $i\in\{1,\ldots,\beta_1\}$, we assume that there exists a subset $\f^\circ_{\Sigma_i}$ of $\f^\circ$ such that $\overline{\Sigma_i}=\bigcup_{F\in\f^\circ_{\Sigma_i}}\overline{F}$, and for which $\normal_{F}=\normal_{\Sigma_i\mid F}$ for all $F\in\f^\circ_{\Sigma_i}$. This requirement ensures that we do associate the tag ``$+$'' to the side of $\hat{\Omega}$ (with respect to $\Sigma_i$) for which $\normal_{\Sigma_i}$ is outward, consistently with our assumption from Section~\ref{sse:top}. Notice that, since the cutting surfaces are piecewise planar, the set $\hat{\Omega}$ is indeed pseudo-Lipschitz. We also let $\f_{\Sigma}^\circ\defi\bigcup_{i=1}^{\beta_1}\f_{\Sigma_i}^\circ$.
When $\beta_2>0$, for all $j\in\{0,\ldots,\beta_2\}$, we let $\f^\partial_{\Gamma_j}$ denote the subset of $\f^\partial$ such that $\Gamma_j=\bigcup_{F\in\f^\partial_{\Gamma_j}}\overline{F}$.

At the discrete level, the parameter $\mu:\Omega\to[\mu_\flat,\mu_\sharp]$ introduced in~\eqref{eq:mu.c} is assumed to be piecewise constant over the partition $\t$ of the domain $\Omega$, and we let
\begin{equation} \label{eq:mu.d}
  \mu_\flat\leq\mu_T\defi\mu_{\mid T}\leq\mu_\sharp\qquad\forall T\in\t.
\end{equation}
The role of this assumption will be made clear in Sections~\ref{ssse:field.d} and~\ref{ssse:vecpot.d} (right before Theorems~\ref{th:field.esti} and~\ref{th:vecpot.esti}).

When studying asymptotic properties with respect to the mesh size, one has to adopt a measure of regularity for refined sequences of discretizations. Let us thus consider $(\d_n)_{n\in\Natu}$, a sequence of polyhedral discretizations $\d_n\defi(\t_n,\f_n)$ such that $h_{\d_n}$ tends to zero as $n$ goes to infinity. We classically follow~\cite[Def.~1.9]{DPDro:20}, in which regularity for refined mesh sequences is quantified through a uniform-in-$n$ parameter $\varrho\in(0,1)$ (the mesh regularity parameter).
In a nutshell, it is assumed that, for all $n\in\Natu$, there exists a matching tetrahedral subtessellation of $\t_n$, (i) which is uniformly-in-$n$ shape-regular, and (ii) whose elements have a diameter that is uniformly-in-$n$ comparable to the diameter of the mesh cell in $\t_n$ they belong to.
In what follows, we write $a\lesssim b$ (resp.~$a\gtrsim b$) in place of $a\leq Cb$ (resp.~$a\geq Cb$), if $C>0$ only depends on $\Omega$, on the mesh regularity parameter $\varrho$, and (if need be) on the underlying polynomial degree, but is independent of both $n$ (and thus $h_{\d_n}$) and $\mu$. When $a\lesssim b\lesssim a$, we simply write $a\eqsim b$.
In particular, for regular sequences $(\d_n)_{n\in\Natu}$ of discretizations, for all $n\in\Natu$ and $T\in\t_n$, there holds ${\rm card}(\f_T)\lesssim 1$, as well as $h_T\lesssim h_F\leq h_T$ for all $F\in\f_T$ (cf.~\cite[Lem.~1.12]{DPDro:20}).
Also, we assume that, for all $n\in\Natu$ and $X\in\f_n\cup\t_n$, the point $\vec{x}_X$ is the center of an $X$-inscribed disk/ball of radius $h_X\lesssim r_X\leq h_X$. This last assumption can always be satisfied for regular sequences of discretizations.

\subsection{Polynomial spaces}

For $\ell\in\Natu$ and $m\in\{2,3\}$, we let $\Poly^\ell_m$ denote the linear space of $m$-variate polynomials of total degree at most $\ell$, with the convention that $\Poly^0_m$ is identified to $\Real$, and that $\Poly^{-1}_m\defi\{0\}$.
For any $X\in\f\cup\t$, we let $\Poly^\ell(X)$ be the linear space spanned by the restrictions to $X$ of the polynomials in $\Poly^\ell_3$. For $X$ of Hausdorff dimension $m\in\{2,3\}$, $\Poly^\ell(X)$ is isomorphic to $\Poly^\ell_m$ (cf.~\cite[Prop.~1.23]{DPDro:20}). We let $\pi^{\ell}_{\Poly,X}$ denote the $\LL[X]$-orthogonal projector onto $\Poly^{\ell}(X)$. For convenience, we also set $\bPoly^\ell(X)\defi\Poly^\ell(X)^m$, that is, $\bPoly^\ell(F)=\Poly^\ell(F)^2$ for all $F\in\f$, and $\bPoly^\ell(T)=\Poly^\ell(T)^3$ for all $T\in\t$. We finally define $\vec{\pi}^{\ell}_{\bPoly,X}$ as the $\bLL[X]$-orthogonal projector onto $\bPoly^{\ell}(X)$.

For any $T\in\t$ and $\ell\in\Natu$, we define $\pGT{\ell}\defi\Grad\big(\Poly^{\ell+1}(T)\big)$, and its Koszul complement $\kGT{\ell}\defi\bPoly^{\ell-1}(T){\times}(\vec{x}-\vec{x}_T)$, both subspaces of $\bPoly^{\ell}(T)$.
The (non $\bLL[T]$-orthogonal) polynomial decomposition below is a by-product of the {\em homotopy formula} (cf.~\cite[Thm.~7.1]{Arnol:18}):
\begin{equation} \label{eq:decomp.T}
  \bPoly^{\ell}(T)=\pGT{\ell}\oplus\kGT{\ell}.
\end{equation}
In addition, by exactness of the polynomial de Rham complex, letting $\pRT{\ell}\defi\Curl\big(\bPoly^{\ell+1}(T)\big)$, the differential mapping $\Curl:\kGT{\ell}\to\pRT{\ell-1}$ is an isomorphism.

For any $F\in\f$ now, we let $H_F$ be the affine hyperplane containing $F$, that we orient according to the normal $\normal_F$. For any function $w:F\to\Real$, we let $\Grad_Fw:F\to\Real^2$ denote the (tangential) gradient of $w$.
We also let $\mathbf{rot}_F w:F\to\Real^2$ be such that $\mathbf{rot}_F w\defi\big(\Grad_Fw\big)^\perp$, where $\vec{z}^\perp$ is defined as the rotation of angle $-\pi/2$ of $\vec{z}$ in the oriented hyperplane $H_F$.
For $\ell\in\Natu$, we define $\pRF{\ell}\defi\mathbf{rot}_F\big(\Poly^{\ell+1}(F)\big)$, and its Koszul complement $\kRF{\ell}\defi\Poly^{\ell-1}(F)(\vec{x}-\vec{x}_F)$, both subspaces of $\bPoly^{\ell}(F)$, where, for $\vec{x}\in F$, $(\vec{x}-\vec{x}_F)\subset H_F$ is identified to its (two-dimensional) tangential counterpart.
By the homotopy formula, in that case again, the following (non $\bLL[F]$-orthogonal) polynomial decomposition holds true:
\begin{equation} \label{eq:decomp.F}
  \bPoly^\ell(F)=\pRF{\ell}\oplus\kRF{\ell}.
\end{equation}
For all $T\in\t$ and $F\in\f_T$, since for any $v:T\to\Real$ we have $\normal_F{\times}((\Grad v)_{\mid F}{\times}\normal_F)=\Grad_F(v_{\mid F})$ (identifying $\normal_F{\times}((\Grad v)_{\mid F}{\times}\normal_F)$ to its two-dimensional proxy), there holds
\begin{equation} \label{eq:rttr}
  \pGT{\ell}_{\mid F}{\times}\normal_F=\pRF{\ell},
\end{equation}
where vectors in $\pGT{\ell}_{\mid F}{\times}\normal_F$ are also identified to their tangential counterparts.
The above identity~\eqref{eq:rttr} provides a characterization for the rotated tangential traces of curl-free vector polynomials.
This characterization is instrumental in the design of stable and optimally consistent trimmed HHO methods (cf.~\cite[Rmk.~10]{LePit:25} for further insight).

Last, we introduce broken versions of the $3$-variate polynomial spaces $\Poly^\ell$ and $\bPoly^{\ell}$, namely $\Poly^\ell(\t)\defi\{v\in\LL\mid v_{\mid T}\in\Poly^\ell(T)\,\forall T\in\t\}$, and $\bPoly^\ell(\t)\defi\{\vec{v}\in\bLL\mid\vec{v}_{\mid T}\in\bPoly^\ell(T)\,\forall T\in\t\}$.
We classically define on $\Poly^\ell(\t)$ the broken gradient operator $\Grad_\t$, and on $\bPoly^{\ell}(\t)$ the broken rotational operator $\Curl_\t$.
We finally let $\ProjTd[\ell]$ (resp.~$\bProjTd[\ell]$) denote the $\LL$-orthogonal (resp.~$\bLL$-orthogonal) projector onto $\Poly^\ell(\t)$ (resp.~$\bPoly^{\ell}(\t)$).

\subsection{Hybrid spaces}

We introduce hybrid counterparts of the spaces $\Hcurl$ and $\Ho[\hat{\Omega}]$ (as well as of relevant subspaces thereof). The first discrete spaces will serve for the discretization of the magnetic variable, whereas the second will be related to the pressure-like variable (Lagrange multiplier).
Henceforth, let $k\in\Natu^\star$ denote a given polynomial degree.

\subsubsection{$\vec{H}(\Curl)$-like hybrid spaces}

\ifSISC
\paragraph{Global spaces}
\else
\paragraph{Global spaces:}
\fi

We define the following discrete counterpart of the space $\Hcurl$:
\begin{equation} \label{eq:dHcurl}
  \dHcurl \defi\left\{ \uvd \defi \Big((\vT)_{T\in\t},(\vFt)_{F\in\f}\Big)\st 
  \begin{alignedat}{2}
    \vT&\in \bPoly^{k}(T) &\quad& \forall T\in\t
    \\
    \vFt&\in \pQF{k} &\quad& \forall F\in\f
  \end{alignedat}
  \right\},
\end{equation}
where the (possibly trimmed) polynomial space $\pQF{k}$ shall satisfy
\begin{equation} \label{eq:pQF}
  \pRF{k}\subseteq\pQF{k}\subseteq\bPoly^{k}(F).
\end{equation}
The requirement $\pRF{k}\subseteq\pQF{k}$ is stability-related (cf.~\cite[Rmk.~10]{LePit:25} for further insight).
The space $\pQF{k}$ may also satisfy, whenever explicitly precised, the additional assumption
\begin{equation} \label{eq:pQF.bis}
  \bPoly^{k-1}(F)\subset\pQF{k}.
\end{equation}
We let $\bProjF$ denote the $\bLL[F]$-orthogonal projector onto $\pQF{k}$. In~\eqref{eq:dHcurl}, $\vFt$ stands for the rotated tangential trace of the (magnetic) variable.
We also introduce the subspace
\begin{equation} \label{eq:dHzcurl}
  \dHzcurl\defi\left\{\uvd\in\dHcurl\mid\vFt\equiv\vec{0}\;\forall F\in\f^{\partial}\right\},
\end{equation}
discrete counterpart of $\Hzcurl$.

\ifSISC
\paragraph{Local traits}
\else
\paragraph{Local traits:}
\fi

Given a cell $T\in\t$, we denote by $\dHcurlT$ the restriction of $\dHcurl$ to $T$, and by $\uvT\defi\big(\vT,(\vFt)_{F\in\f_T}\big)\in\dHcurlT$ the restriction of the generic element $\uvd\in\dHcurl$.
For $\uvd\in\dHcurl$, we also let $\vTd$ be the broken polynomial vector field in $\bPoly^k(\t)$ such that $\vec{v}_{\t\mid T} \defi \vT$ for all $T\in\t$.
Let us now define an $\vec{H}(\Curl)$-like hybrid semi-norm on $\dHcurlT$. We let, for all $\uvT\in\dHcurlT$,
\begin{equation} \label{eq:dHcurlT.sn}
  |\uvT|_{\Curl,T}^2\defi\|\Curl\vT\|_{0,T}^2+\sum_{F\in\f_T}h_F^{-1}\|\bProjF(\vec{v}_{T\mid F}{\times}\normal_F)-\vFt\|_{0,F}^2,
\end{equation}
where $\vec{v}_{T\mid F}{\times}\normal_F$ is identified to its tangential proxy.
In turn, at the global level, we classically set, for all $\uvd\in\dHcurl$, $|\uvd|_{\Curl,\d}^2\defi\sum_{T\in\t}|\uvT|_{\Curl,T}^2$.
Remark that, whenever $|\uvd|_{\Curl,\d}=0$ for some $\uvd\in\dHcurl$, then, by~\eqref{eq:rttr} and~\eqref{eq:pQF}, one has $\vTd\in\Hcurlz\cap\vec{\mathcal{G}}^k(\t)$.
We then define the rotational reconstruction operator. For any $T\in\t$, we let $\CT:\dHcurlT\to\bPoly^{k-1}(T)$ be the operator such that, for all $\uvT\in\dHcurlT$, $\CT(\uvT)\in\bPoly^{k-1}(T)$ is the unique solution to
\begin{equation} \label{eq:CT}
  \big(\CT(\uvT),\vec{z}\big)_T = (\vT,\Curl\vec{z})_T-\!\sum_{F\in\f_T}\!\varepsilon_{T,F}\big(\vFt,\normal_F{\times}(\vec{z}_{\mid F}{\times}\normal_F)\big)_F\qquad\forall\vec{z}\in\bPoly^{k-1}(T),
\end{equation}
where $\normal_F{\times}(\vec{z}_{\mid F}{\times}\normal_F)$ is identified to its tangential proxy.
At the global level, we let $\CTd:\dHcurl\to\bPoly^{k-1}(\t)$ be such that, for all $\uvd\in\dHcurl$, $\CTd(\uvd)_{\mid T}\defi\CT(\uvT)$ for all $T\in\t$.
It now remains to define an $\vec{H}(\Curl)$-like hybrid stabilizer. For any $T\in\t$, we introduce the following symmetric, positive semi-definite bilinear form: for all $\uvT,\uwT\in\dHcurlT$,
\begin{equation} \label{eq:dHcurlT.stab}
  S_{\Curl,T}(\uwT,\uvT)\defi\sum_{F\in\f_T}h_F^{-1}\big(\bProjF\big(\vec{w}_{T\mid F}{\times}\normal_F\big) - \wFt,\bProjF\big(\vec{v}_{T\mid F}{\times}\normal_F\big) - \vFt\big)_F.
\end{equation}
At the global level, we then classically let $S_{\Curl,\d}(\uwd,\uvd)\defi\sum_{T\in\t}S_{\Curl,T}(\uwT,\uvT)$ for all $\uvd,\uwd\in\dHcurl$.
Given $T\in\t$, under the assumption that $\pQF{k}$ (already satisfying~\eqref{eq:pQF}) additionally satisfies~\eqref{eq:pQF.bis}, it is an easy matter to prove that, for all $\uvT\in\dHcurlT$,
\begin{equation} \label{eq:CT.equiv}
  \|\CT(\uvT)\|_{0,T}^2+S_{\Curl,T}(\uvT,\uvT)\eqsim|\uvT|_{\Curl,T}^2.
\end{equation}

\ifSISC
\paragraph{Approximation}
\else
\paragraph{Approximation:}
\fi

Let $q>2$. Given $T\in\t$, we introduce the local reduction operator $\intHcurlT:\Hcurl[T]\cap\vec{L}^q(T)\to\dHcurlT$ such that, for any $\vec{v}\in\Hcurl[T]\cap\vec{L}^q(T)$,
\begin{equation} \label{eq:dHcurlT.int}
    \intHcurlT(\vec{v})\defi\left(\bProjT(\vec{v}),\big(\bProjF(\vec{v}_{\mid F}{\times}\normal_F)\big)_{F\in\f_T}\right),
\end{equation}
where, as now standard, $\vec{v}_{\mid F}{\times}\normal_F$ is identified to its (two-dimensional) tangential counterpart.
Following~\cite[Sec.~5.1]{ErnGu:22}, note that the regularity $\vec{v}\in\Hcurl[T]\cap\vec{L}^q(T)$ is sufficient to give a (weak) meaning to the face polynomial projections in~\eqref{eq:dHcurlT.int} (see also~\cite[Rmk.~11]{LePit:25}).
At the domain level, now, the global reduction operator $\intHcurl:\Hcurl\cap\vec{L}^q(\Omega)\to\dHcurl$ is defined so that, for any $\vec{v}\in\Hcurl\cap\vec{L}^q(\Omega)$, $\intHcurl(\vec{v})\defi\big(\big(\bProjT(\vec{v}_{\mid T})\big)_{T\in\t},\big(\bProjF(\vec{v}_{\mid F}{\times}\normal_F)\big)_{F\in\f}\big)$.
Remark that, since $\vec{v}\in\Hcurl$, the quantity $\vec{v}_{\mid F}{\times}\normal_F$ is single-valued at interfaces $F\in\f^\circ$.
Notice also that $\intHcurl\big(\Hzcurl\cap\vec{L}^q(\Omega)\big)\subset\dHzcurl$.
Given $T\in\t$, under the assumption that $\pQF{k}$ (already satisfying~\eqref{eq:pQF}) additionally satisfies~\eqref{eq:pQF.bis}, it is an easy matter to prove the following commutation property: for all $\vec{v}\in\Hcurl[T]\cap\vec{L}^q(T)$,
\begin{equation} \label{eq:dHcurlT.comm}
  \big(\CT\circ\intHcurlT\big)(\vec{v})=\bProjT[k-1](\Curl\vec{v}).
\end{equation}
\ifSISC
\else
Finally, without this time the need for the additional assumption~\eqref{eq:pQF.bis} on $\pQF{k}$, the following (optimal) polynomial consistency result holds true: for all $\vec{p}\in\bPoly^k(T)$,
\begin{equation*}
  S_{\Curl,T}\big(\intHcurlT(\vec{p}),\uvT\big)=0\qquad\forall\uvT\in\dHcurlT.
\end{equation*}
\fi

\subsubsection{$H^1$-like hybrid spaces}

\ifSISC
\paragraph{Global spaces}
\else
\paragraph{Global spaces:}
\fi

We define the following discrete counterpart of the space $\Ho[\hat{\Omega}]$:
\begin{equation} \label{eq:hdHo}
  \hdHo \defi\left\{ \uqd \defi \Big((\qT)_{T\in\t},(\qF)_{F\in\f\setminus\f^\circ_{\Sigma}},(\qF^+,\qF^-)_{F\in\f^\circ_{\Sigma}}\Big)\st 
  \begin{alignedat}{2}
    \qT&\in \Poly^{k-1}(T) &\quad& \forall T\in\t
    \\
    \qF&\in \Poly^{k}(F) &\quad& \forall F\in\f\setminus\f^\circ_{\Sigma}\\
    \qF^+,\qF^-&\in \Poly^{k}(F) &\quad& \forall F\in\f^\circ_{\Sigma}
  \end{alignedat}
  \right\},
\end{equation}
where we recall that $\f_{\Sigma}^\circ=\bigcup_{i=1}^{\beta_1}\f_{\Sigma_i}^\circ$, with $\f_{\Sigma_i}^\circ$ subset of $\f^\circ$ such that $\overline{\Sigma_i}=\bigcup_{F\in\f_{\Sigma_i}^\circ}\overline{F}$.
Here, for any $i\in\{1,\ldots,\beta_1\}$ and $F\in\f^\circ_{\Sigma_i}$, $\qF^+$ and $\qF^-$ stand for the traces on $F$ of the pressure-like variable defined, respectively, from the ``$+$'' and ``$-$'' sides of $\hat{\Omega}$ with respect to the cutting surface $\Sigma_i$. Mirroring~\eqref{eq:jump}, for $\uqd\in\hdHo$, we define its jump through $F\in\f^\circ_{\Sigma_i}$ by $\llbracket \uqd\rrbracket_F\defi\qF^+-\qF^-$.
We then introduce the following subspace of $\hdHo$:
\begin{equation} \label{eq:hdHSof}
  \hdHSof\defi\left\{\uqd\in\hdHo\st\exists\,(\sigma_i)\in\Real^{\beta_1},\,\llbracket\uqd\rrbracket_F\equiv\sigma_i\;\forall F\in\f^\circ_{\Sigma_i},\,i\in\{1,\ldots,\beta_1\}\right\},
\end{equation}
which is, modulo the zero-mean condition over $\hat{\Omega}$, an hybrid counterpart of the space $\HSo$ defined in~\eqref{eq:HSo}. Next, we define the following discrete counterpart of the space $\Ho$:
\begin{equation} \label{eq:dHo}
  \dHo\defi\left\{\uqd\defi\Big((\qT)_{T\in\t},(\qF)_{F\in\f}\Big)\st 
  \begin{alignedat}{2}
    \qT&\in \Poly^{k-1}(T) &\quad& \forall T\in\t
    \\
    \qF&\in \Poly^{k}(F) &\quad& \forall F\in\f
  \end{alignedat}
  \right\}.
\end{equation}
Remark that if $\beta_1=0$, then $\hdHSof=\hdHo=\dHo$.
Whenever $\beta_1>0$, the space $\dHo$ is isomorphic to the subspace of $\hdHSof$ such that $\llbracket\uqd\rrbracket_F\equiv 0$ for all $F\in\f^\circ_\Sigma$.
We may now introduce the following subspaces of $\dHo$:
\begin{equation} \label{eq:dHzo}
  \dHzo\defi\left\{\uqd\in\dHo\st\qF\equiv 0\;\forall F\in\f^\partial_{\Gamma_0}\right\},
\end{equation}
as well as
\begin{equation} \label{eq:dHGo}
  \dHGo\defi\left\{\uqd\in\dHzo\st\exists\,(\gamma_j)\in\Real^{\beta_2},\,\qF\equiv\gamma_j\;\forall F\in\f^\partial_{\Gamma_j},\,j\in\{1,\ldots,\beta_2\}\right\}.
\end{equation}
The space $\dHGo$ is the hybrid counterpart of the space $\HGo$ in~\eqref{eq:HGo}. Remark that, whenever $\beta_2=0$, then $\dHGo=\dHzo$ with $\dHzo$ being, in that case, the hybrid counterpart of $\Hzo$.

\ifSISC
\paragraph{Local traits}
\else
\paragraph{Local traits:}
\fi

Given a cell $T\in\t$, we denote by $\dHoT$ the restriction of $\hdHo$ or $\dHo$ to $T$, and by $\uqT\defi\big(\qT,(q_{F,T})_{F\in\f_T}\big)\in\dHoT$ the restriction of the generic element $\uqd$ in $\hdHo$ or $\dHo$. For $\uqd\in\dHo$, $q_{F,T}\defi\qF$ for all $F\in\f_T$, whereas for $\uqd\in\hdHo$,
\begin{itemize}
  \item[$\bullet$] if $F\in\f_T\setminus\f^\circ_{\Sigma}$, then $q_{F,T}\defi\qF$;
  \item[$\bullet$] if $F\in\f_T\cap\f^\circ_{\Sigma}$ with $F\subset\partial T^+\cap\partial T^-$, then $q_{F,T}\defi\qF^+$ if $T=T^+$, whereas $q_{F,T}\defi\qF^-$ if $T=T^-$.
\end{itemize}
For $\uqd$ in $\hdHo$ or $\dHo$, we also let $\qTd$ be the broken polynomial function in $\Poly^{k-1}(\t)$ such that $q_{\t\mid T} \defi \qT$ for all $T\in\t$.
Let us now define an ($L^2$-scaled) $H^1$-like hybrid semi-norm on $\dHoT$. We let, for all $\uqT\in\dHoT$,
\begin{equation} \label{eq:dHoT.sn}
  |\uqT|_{\Grad,T}^2\defi h_T^2\|\Grad\qT\|_{0,T}^2+\sum_{F\in\f_T}h_F\|q_{T\mid F}-q_{F,T}\|_{0,F}^2.
\end{equation}
In turn, at the global level, for all $\uqd$ in $\hdHo$ or $\dHo$, we set $|\uqd|_{\Grad,\d}^2\defi\sum_{T\in\t}|\uqT|_{\Grad,T}^2$.
Remark that, whenever $|\uqd|_{\Grad,\d}=0$ for some $\uqd$ in $\hdHo$ or $\dHo$, then there exists $c\in\Real$ such that $\uqd\cong c\underline{1}_\d$, with $\underline{1}_\d$ the vector of $\dHo$ whose components are all equal to one.
We now define the gradient reconstruction operator. For any $T\in\t$, we let $\GT:\dHoT\to\bPoly^{k}(T)$ be the operator such that, for all $\uqT\in\dHoT$, $\GT(\uqT)\in\bPoly^{k}(T)$ is the unique solution to
\begin{equation} \label{eq:GT}
  \big(\GT(\uqT),\vec{z}\big)_T = -(\qT,\Div\vec{z})_T+\!\sum_{F\in\f_T}\big(q_{F,T},\vec{z}_{\mid F}{\cdot}\normal_{T,F}\big)_F\qquad\forall\vec{z}\in\bPoly^{k}(T).
\end{equation}
At the global level, we let $\GTd:\big\{\hdHo,\dHo\big\}\to\bPoly^{k}(\t)$ be such that, for all $\uqd$ in $\hdHo$ or $\dHo$, $\GTd(\uqd)_{\mid T}\defi\GT(\uqT)$ for all $T\in\t$.
It now remains to define an ($L^2$-scaled) $H^1$-like hybrid stabilizer. For any $T\in\t$, we introduce the following symmetric, positive semi-definite bilinear form: for all $\uqT,\urT\in\dHoT$,
\begin{equation} \label{eq:dHoT.stab}
  S_{\Grad,T}(\urT,\uqT)\defi\sum_{F\in\f_T}h_F\big(r_{T\mid F}-r_{F,T},q_{T\mid F}-q_{F,T}\big)_F.
\end{equation}
At the global level, we then set $S_{\Grad,\d}(\urd,\uqd)\defi\sum_{T\in\t}S_{\Grad,T}(\urT,\uqT)$ for all $\uqd,\urd$ in $\hdHo$ or $\dHo$.
Given $T\in\t$, it is an easy matter to prove that, for all $\uqT\in\dHoT$,
\begin{equation} \label{eq:GT.equiv}
  h_T^2\|\GT(\uqT)\|_{0,T}^2+S_{\Grad,T}(\uqT,\uqT)\eqsim|\uqT|_{\Grad,T}^2.
\end{equation}
Besides, under the assumption that $T$ belongs to a tetrahedral mesh sequence, by~\cite[Lem.~1]{CDPLe:22} (and adapting the arguments from the proof of~\cite[Lem.~3.2]{CEPig:21}), the following important equivalence holds true: for all $\uqT\in\dHoT$,
\begin{equation} \label{eq:GT.coer}
  h_T^2\|\GT(\uqT)\|_{0,T}^2\eqsim|\uqT|_{\Grad,T}^2.
\end{equation}

\ifSISC
\paragraph{Approximation}
\else
\paragraph{Approximation:}
\fi

Given $T\in\t$, we introduce the local reduction operator $\intHoT:\Ho[T]\to\dHoT$ such that, for any $q\in\Ho[T]$,
\begin{equation} \label{eq:dHoT.int}
    \intHoT(q)\defi\left(\ProjT(q),\big(\ProjF(q_{\mid F})\big)_{F\in\f_T}\right).
\end{equation}
At the domain level, now, the global reduction operator $\intHo:\big\{\Ho[\hat{\Omega}],\Ho\big\}\to\big\{\hdHo,\dHo\big\}$ is defined in the following way:
\begin{itemize}
  \item[$\bullet$] for any $q\in\Ho[\hat{\Omega}]$,
    \begin{equation*}
      \intHo(q)\defi\left(\big(\ProjT(q_{\mid T})\big)_{T\in\t},\big(\ProjF(q_{\mid F})\big)_{F\in\f\setminus\f^\circ_{\Sigma}},\big(\ProjF(q^+_{\mid F}),\ProjF(q^-_{\mid F})\big)_{F\in\f^\circ_{\Sigma}}\right),
    \end{equation*}
    where, for $F\in\f^\circ_{\Sigma}$ such that $F\subset\partial T^+\cap\partial T^-$, we have let $q^+\defi q_{\mid T^+}$ and $q^-\defi q_{\mid T^-}$;
  \item[$\bullet$] for any $q\in\Ho$, $\intHo(q)\defi\big(\big(\ProjT(q_{\mid T})\big)_{T\in\t},\big(\ProjF(q_{\mid F})\big)_{F\in\f}\big)$.
\end{itemize}
Remark that, for $F\in\f^\circ$ with $F\subset\partial T^+\cap\partial T^-$, whenever $q\in H^1\big(\overline{T^+}\cup\overline{T^-}\big)$, then the quantity $q_{\mid F}$ is single-valued.
Note, in addition, that $\intHo\big(\HSo\big)\subset\hdHSof$ and that $\intHo\big(\HGo\big)\subset\dHGo$.
Finally, given $T\in\t$, the following commutation property is valid: for all $q\in\Ho[T]$,
\begin{equation} \label{eq:dHoT.comm}
  \big(\GT\circ\intHoT\big)(q)=\bProjT(\Grad q).
\end{equation}
\ifSISC
\else
Also, the following polynomial consistency result holds true: for all $p\in\Poly^{k-1}(T)$,
\begin{equation*}
  S_{\Grad,T}\big(\intHoT(p),\uqT\big)=0\qquad\forall\uqT\in\dHoT.
\end{equation*}
\fi

\section{HHO methods} \label{se:hho}

We consider two models of magnetostatics: (i) the (first-order) field formulation, endowed with normal boundary conditions, and (ii) the (second-order) vector potential formulation, endowed with tangential boundary conditions. This last model was already considered in~\cite[Sec.~3.2]{CDPLe:22}, but therein for trivial topology only. These two problems are equivalent. They correspond to the primal instance (that is, based on the primal de Rham complex) of the model problems~\eqref{eq:prob.fo} and~\eqref{eq:prob.so}. The dual instance could be handled similarly.

\subsection{Magnetostatics models}

Recall the definitions~\eqref{eq:HGo} and~\eqref{eq:HSo} of the spaces $\HGo$ and $\HSo$.
Henceforth, the (variable) coefficient $\mu$, satisfying~\eqref{eq:mu.d}, will denote the magnetic permeability of the medium.

\subsubsection{Field formulation} \label{ssse:field}

Let $\vec{j}:\Omega\to\Real^3$ be a given electric current density satisfying the following compatibility conditions:
\begin{equation} \label{eq:comp.str}
  \Div\vec{j}=0\;\text{in $\Omega$},\qquad\langle\vec{j}_{\mid\Gamma_j}{\cdot}\normal,1\rangle_{\Gamma_j}=0\;\text{for all $j\in\{1,\ldots,\beta_2\}$}.
\end{equation}
We seek the magnetic field $\vec{h}:\Omega\to\Real^3$ such that
\begin{subequations}
  \label{eq:field.str}
  \begin{alignat}{2}
    \Curl \vec{h} &= \vec{j} &\qquad&\text{in $\Omega$},\label{eq:field.str.a}
    \\
    \Div \vec{b} &= 0 &\qquad&\text{in $\Omega$},\label{eq:field.str.b}
	\\
    \vec{b}_{\mid\Gamma}{\cdot}\normal  &=  0 &\qquad&\text{on $\Gamma$},\label{eq:field.str.c}
    \\
    \langle\vec{b}_{\mid\Sigma_i}{\cdot}\normal_{\Sigma_i},1\rangle_{\Sigma_i}&=0&\qquad&\text{for all $i\in\{1,\ldots,\beta_1\}$},\label{eq:field.str.d}
  \end{alignat}
\end{subequations}
with constitutive law $\vec{b}=\mu\vec{h}$, where $\vec{b}:\Omega\to\Real^3$ represents the magnetic induction.

In what follows, we assume that $\vec{j}\in\bLL$. By~\eqref{eq:GradHGo}, remark that~\eqref{eq:comp.str} is equivalent to
\begin{equation} \label{eq:comp.wea}
  (\vec{j},\vec{z})_{\Omega}=0\qquad\forall\vec{z}\in\Grad\big(\HGo\big).
\end{equation}
In the spirit of Kikuchi~\cite{Kikuc:89}, we consider the following equivalent (cf.~Remark~\ref{re:wse} below) weak form for Problem~\eqref{eq:field.str}: Find $(\vec{h},p)\in\Hcurl\times\HSo$ s.t.
\begin{subequations}\label{eq:field.wea}
  \begin{alignat}{2}
    A(\vec{h},\vec{v}) + B(\vec{v},p) & = (\vec{j},\Curl \vec{v})_\Omega &\qquad&\forall \vec{v}\in\Hcurl,\label{eq:field.wea.a}
    \\
    -B(\vec{h},q) & = 0 &\qquad&\forall q\in\HSo,\label{eq:field.wea.b}
  \end{alignat}
\end{subequations}
with bilinear forms $A:\Hcurl\times\Hcurl\rightarrow\Real$, $B:\Hcurl\times\Ho[\hat{\Omega}]\rightarrow\Real$ set to
\begin{equation} \label{eq:field.bilfor}
  A(\vec{w},\vec{v}) \defi (\Curl \vec{w},\Curl \vec{v})_\Omega,
  \qquad
  B(\vec{w},q) \defi (\mu\vec{w},\check{\Grad q})_\Omega,
\end{equation}
where we recall that, for $q\in\Ho[\hat{\Omega}]$, $\check{\Grad}\,q$ is the continuation to $\bLL$ of $\Grad q\in\bLL[\hat{\Omega}]$.
The pressure-like variable $p$ is the Lagrange multiplier of the constraints~\eqref{eq:field.str.b}--\eqref{eq:field.str.d} on the magnetic induction $\vec{b}$.
Testing~\eqref{eq:field.wea.a} with $\vec{v}=\check{\Grad p}\in\check{\Grad}\big(H^1_{\Sigma}(\hat{\Omega})\big)\subset\Hcurlz$, and recalling that $p\in L^2_0(\hat{\Omega})$, one can actually infer that $p=0$ in $\hat{\Omega}$ (recall that $\hat{\Omega}$ is assumed to be connected).
The well-posedness of Problem~\eqref{eq:field.wea} is a direct consequence of the second Weber inequality, more precisely of~\cite[Rmk.~6]{LePit:25} with $\eta\defi\mu$ (combined with~\eqref{eq:field.wea.b} and~\eqref{eq:GradHSo}).

\begin{remark}[Weak-strong equivalence] \label{re:wse}
  Whereas it is clear that any solution to Problem~\eqref{eq:field.str} also solves Problem~\eqref{eq:field.wea}, the converse is less straightforward, in particular when it comes to retrieving~\eqref{eq:field.str.a}.
  To do so, one has to use the first Helmholtz--Hodge decomposition~\eqref{eq:helm1} (with $\mu\leftarrow 1$ therein), combine it to~\eqref{eq:GradHGo}, then recall the compatibility condition~\eqref{eq:comp.wea}, so as to infer that $\big(\Curl\vec{h}-\vec{j},\vec{z}\big)_{\Omega}=0$ for all $\vec{z}\in\bLL$, thereby yielding~\eqref{eq:field.str.a}.
\end{remark}

\begin{remark}[Regularity theory] \label{re:reg.field}
  Recall that $\Omega$ is a Lipschitz polyhedron.
  When the coefficient $\mu$ is globally smooth (which, under~\eqref{eq:mu.d}, amounts to assuming that $\mu$ is constant in $\Omega$), it is known (cf.~\cite[Thm.~2]{Costa:90} and~\cite[Prop.~3.7]{ABDGi:98}) that the solution to Problem~\eqref{eq:field.str} (and more generally to Problem~\eqref{eq:prob.fo}) satisfies $\vec{h}\in\vec{H}^s(\Omega)$ for some $s>\frac{1}{2}$, with $s\geq 1$ if $\Omega$ is convex.
  When the coefficient $\mu$ is piecewise smooth (here, under~\eqref{eq:mu.d}, piecewise constant), however, the solution $\vec{h}$ to Problem~\eqref{eq:field.str} (and, more generally, to Problem~\eqref{eq:prob.fo}) is only known (cf.~\cite{CoDaN:99,Jochm:99,BGLud:13}) to belong to $\vec{H}^s(\Omega)$ for some $s>0$.
\end{remark}

\subsubsection{Vector potential formulation} \label{ssse:vecpot}

Let $\vec{j}:\Omega\to\Real^3$ be a given electric current density complying with~\eqref{eq:comp.str}.
We seek the magnetic vector potential $\vec{a}:\Omega\to\Real^3$ such that
\begin{subequations}
  \label{eq:vecpot.str}
  \begin{alignat}{2}
    \Curl(\mu^{-1}\Curl \vec{a}) &= \vec{j} &\qquad&\text{in $\Omega$},\label{eq:vecpot.str.a}
    \\
    \Div \vec{a} &= 0 &\qquad&\text{in $\Omega$},\label{eq:vecpot.str.b}
	\\
    \vec{a}_{\mid\Gamma}{\times}\normal  &=  \vec{0} &\qquad&\text{on $\Gamma$},\label{eq:vecpot.str.c}
    \\
    \langle\vec{a}_{\mid\Gamma_j}{\cdot}\normal,1\rangle_{\Gamma_j}&=0&\qquad&\text{for all $j\in\{1,\ldots,\beta_2\}$}.\label{eq:vecpot.str.d}
  \end{alignat}
\end{subequations}

In what follows, we assume that $\vec{j}\in\bLL$. Recall that~\eqref{eq:comp.str} is equivalent to~\eqref{eq:comp.wea}.
We consider the following weak form for Problem~\eqref{eq:vecpot.str}: Find $(\vec{a},p)\in\Hzcurl\times\HGo$ s.t.
\begin{subequations}\label{eq:vecpot.wea}
	\begin{alignat}{2}
		A(\vec{a},\vec{v}) + B(\vec{v},p) & = (\vec{j},\vec{v})_\Omega &\qquad&\forall \vec{v}\in\Hzcurl,\label{eq:vecpot.wea.a}
		\\
		-B(\vec{a},q) & = 0 &\qquad&\forall q\in\HGo,\label{eq:vecpot.wea.b}
	\end{alignat}
\end{subequations}
with bilinear forms $A:\Hcurl\times\Hcurl\rightarrow\Real$, $B:\Hcurl\times\Ho\rightarrow\Real$ set to
\begin{equation} \label{eq:vecpot.bilfor}
  A(\vec{w},\vec{v}) \defi (\mu^{-1}\Curl \vec{w},\Curl \vec{v})_\Omega,
  \qquad
  B(\vec{w},q) \defi (\vec{w},\Grad q)_\Omega.
\end{equation}
The pressure-like variable $p$ is the Lagrange multiplier of the constraints~\eqref{eq:vecpot.str.b} and~\eqref{eq:vecpot.str.d} on the vector potential $\vec{a}$.
Testing~\eqref{eq:vecpot.wea.a} with $\vec{v}=\Grad p\in\Grad\big(H^1_{\Gamma}(\Omega)\big)\subset\Hzcurlz$, using~\eqref{eq:comp.wea}, and recalling the definition~\eqref{eq:HGo}, one can actually infer that $p=0$ in $\Omega$. Consequently, Problems~\eqref{eq:vecpot.str} and~\eqref{eq:vecpot.wea} are equivalent (that is, weak-strong equivalence holds true).
The well-posedness of Problem~\eqref{eq:vecpot.wea} is a direct consequence of the first Weber inequality, more precisely of~\cite[Rmk.~3]{LePit:25} with $\eta\defi 1$ (combined with~\eqref{eq:vecpot.wea.b} and~\eqref{eq:GradHGo}).

\begin{remark}[Equivalence with Problem~\eqref{eq:field.str}]
  Whereas it is clear that, for any solution $\vec{a}$ to Problem~\eqref{eq:vecpot.str}, setting $\vec{b}\defi\Curl\vec{a}$, then $\vec{h}=\mu^{-1}\vec{b}$ is solution to Problem~\eqref{eq:field.str}, the converse is less obvious.
  It is a consequence of the fact that the equations~\eqref{eq:field.str.b}--\eqref{eq:field.str.d} imply, by the second Helmholtz--Hodge decomposition~\eqref{eq:helm2}, the existence of $\vec{a}\in\Hzcurl$ such that $\vec{h}=\mu^{-1}\Curl\vec{a}$.
  The latter vector potential is non-unique, which is the reason why one has to further impose a gauge condition.
  In our case, we impose the so-called Coulomb gauge, which amounts to enforcing~\eqref{eq:vecpot.str.b} and~\eqref{eq:vecpot.str.d}.
\end{remark}

\begin{remark}[Regularity theory] \label{re:reg.vecpot}
  Given that $\vec{h}=\mu^{-1}\Curl\vec{a}$ is solution to Problem~\eqref{eq:field.str}, we infer from Remark~\ref{re:reg.field} that $\mu^{-1}\Curl\vec{a}\in\vec{H}^r(\Omega)$, with $r>\frac{1}{2}$ if $\mu$ is globally smooth, or $r>0$ if $\mu$ is piecewise smooth. In turn, since $\vec{a}\in\Hzcurl\cap\vec{H}(\Div^0;\Omega)$, there holds $\vec{a}\in\vec{H}^s(\Omega)$ for some $s>\frac{1}{2}$. In particular, if $\mu$ is globally constant, both $\mu^{-1}\Curl\vec{a}$ and $\vec{a}$ belong to $\vec{H}^t(\Omega)$ for some $t>\frac{1}{2}$, with $t\geq 1$ if $\Omega$ is convex. In the general case, the solution to Problem~\eqref{eq:vecpot.str} (and, more generally, to Problem~\eqref{eq:prob.so}) satisfies $\vec{a}\in\vec{H}^s(\Omega)$ for some $s>\frac{1}{2}$, and $\mu^{-1}\Curl\vec{a}\in\vec{H}^r(\Omega)$ for some $r>0$.
\end{remark}

\subsection{HHO schemes} \label{sse:hho}

The HHO methods we devise are directly built upon the weak formulations~\eqref{eq:field.wea} and~\eqref{eq:vecpot.wea}. At the discrete level, we also heavily leverage the fact that the continuous Lagrange multipliers are equal to zero.

\subsubsection{Field formulation} \label{ssse:field.d}

Recall the definition~\eqref{eq:dHcurl} of the space $\dHcurl$, hybrid counterpart of $\Hcurl$, and assume that, for all $F\in\f$, the space $\pQF{k}$ is given by
\begin{equation} \label{eq:pQF.choice}
  \pQF{k}\defi\pRF{k}\oplus\kRF{k-1},
\end{equation}
in such a way that~\eqref{eq:pQF} is indeed fulfilled, and that~\eqref{eq:pQF.bis} additionally holds true.
Recall the definition~\eqref{eq:hdHSof} of the space $\hdHSof$. We introduce the following hybrid counterpart of $\HSo$:
\begin{equation} \label{eq:hdHSo}
  \hdHSo\defi\left\{\uqd\in\hdHSof\st\int_{\hat{\Omega}}\qTd=0\right\}.
\end{equation}
Remark that $\intHo\big(\HSo\big)\subset\hdHSo$, since functions in $\HSo$ also belong to $L^2_0(\hat{\Omega})$.

Given $T\in\t$, we introduce the local bilinear forms $A_T:\dHcurlT\times\dHcurlT\to\Real$, $B_T:\dHcurlT\times\dHoT\to\Real$, and $S_T:\dHoT\times\dHoT\to\Real$ defined by: for all $\uvT,\uwT\in\dHcurlT$, and $\uqT,\urT\in\dHoT$,
\begin{subequations}
  \label{eq:field.dbf}
  \begin{alignat}{1}
    A_T(\uwT,\uvT)&\defi\big(\CT(\uwT),\CT(\uvT)\big)_T+S_{\Curl,T}(\uwT,\uvT),\label{eq:field.dbf.A}
    \\
    B_T(\uwT,\uqT)&\defi\mu_T\big(\wT,\GT(\uqT)\big)_T,\label{eq:field.dbf.B}
    \\
    S_T(\urT,\uqT)&\defi \mu_T^2h_T^2\big(\Grad\rT,\Grad\qT\big)_T+\mu_T^2S_{\Grad,T}(\urT,\uqT),\label{eq:field.dbf.N}
  \end{alignat}
\end{subequations}
where the local stabilizers $S_{\Curl,T}$ and $S_{\Grad,T}$ are respectively defined by~\eqref{eq:dHcurlT.stab} and~\eqref{eq:dHoT.stab}.
The scaling $\mu_T^2$ (also $h_T^2$) in~\eqref{eq:field.dbf.N} is chosen so as to restore consistency between the physical units.
As far as the bilinear contribution $S_T$ is concerned, its role  will be discussed in Remark~\ref{re:tet} below.
As now standard, the global bilinear forms $A_\d:\dHcurl\times\dHcurl\to\Real$, $B_\d:\dHcurl\times\hdHo\to\Real$, and $S_\d:\hdHo\times\hdHo\to\Real$ are assembled summing the local contributions.
The discrete HHO problem reads: find $(\uhd,\upd)\in\dHcurl\times\hdHSo$ s.t.
\begin{subequations}\label{eq:field.dis}
  \begin{alignat}{2}
    A_\d(\uhd,\uvd) + B_\d(\uvd,\upd) & = \big(\vec{j},\CTd(\uvd)\big)_\Omega &\qquad&\forall \uvd\in\dHcurl,\label{eq:field.dis.a}
    \\
    -B_\d(\uhd,\uqd) + S_\d(\upd,\uqd) & = 0 &\qquad&\forall\uqd\in\hdHSo.\label{eq:field.dis.b}
  \end{alignat}
\end{subequations}
We now establish well-posedness for Problem~\eqref{eq:field.dis}.
\begin{lemma}[Well-posedness] \label{le:field.wp}
  For all $\uvd\in\dHcurl$ and $\uqd\in\hdHo$, there holds:
  \begin{equation} \label{eq:field.wp}
    |\uvd|_{\Curl,\d}^2+\sum_{T\in\t}\mu_T^2|\uqT|_{\Grad,T}^2\lesssim A_\d(\uvd,\uvd)+S_\d(\uqd,\uqd).
  \end{equation}
  As a consequence, Problem~\eqref{eq:field.dis} is well-posed.
\end{lemma}
\begin{proof}
  Let $\uvT\in\dHcurlT$ and $\uqT\in\dHoT$.
  Starting from definition~\eqref{eq:field.dbf.A}, and invoking~\eqref{eq:CT.equiv} (which is valid under assumption~\eqref{eq:pQF.choice}), there holds
  \begin{equation} \label{eq:field.AT}
    A_T(\uvT,\uvT)=\|\CT(\uvT)\|_{0,T}^2+S_{\Curl,T}(\uvT,\uvT)\eqsim|\uvT|_{\Curl,T}^2.
  \end{equation}
  In the same vein, by~\eqref{eq:field.dbf.N}, $S_T(\uqT,\uqT)=\mu_T^2|\uqT|_{\Grad,T}^2$.
  Summing over $T\in\t$ yields~\eqref{eq:field.wp}.

  Now, since the linear system corresponding to Problem~\eqref{eq:field.dis} is square, proving well-posedness is equivalent to proving the uniqueness of its solution. Let us show that, if $\vec{j}\equiv\vec{0}$, then necessarily $(\uhd,\upd)=\big(\underline{\vec{0}}_\d,\underline{0}_\d\big)$. First, remark that the map $|{\cdot}|_{\Grad,\d}$ (cf.~\eqref{eq:dHoT.sn}) is a norm on $\hdHSo$ (cf.~\eqref{eq:hdHSo}), as a by-product of the discrete zero-mean condition over $\hat{\Omega}$. Second, let us test Problem~\eqref{eq:field.dis} (for $\vec{j}\equiv\vec{0}$) with $\uvd=\uhd$ and $\uqd=\upd$. Summing~\eqref{eq:field.dis.a} and~\eqref{eq:field.dis.b}, and leveraging~\eqref{eq:field.wp}, we directly infer that $|\uhd|_{\Curl,\d}=0$ and $|\upd|_{\Grad,\d}=0$.
  From the second, since $\upd\in\hdHSo$, we deduce that $\upd=\underline{0}_\d$. Plugging this value into~\eqref{eq:field.dis.b}, we get $B_\d(\uhd,\uqd)=\big(\mu\hTd,\GTd(\uqd)\big)_{\Omega}=0$ for all $\uqd\in\hdHSo$.
  Since, for any $q\in\HSo$, $\intHo(q)\in\hdHSo$, letting $\uqd=\intHo(q)$ in the latter relation, and invoking the local commutation property~\eqref{eq:dHoT.comm}, together with $\mu\hTd\in\bPoly^k(\t)$ (as a by-product of~\eqref{eq:mu.d}), we infer
  \begin{equation} \label{eq:stabi}
    \big(\mu\hTd,\vec{z}\big)_{\Omega}=0\qquad\forall\vec{z}\in\check{\Grad}\big(\HSo\big).
  \end{equation}
  The conclusion then follows from the second hybrid Weber inequality of~\cite[Rmk.~16]{LePit:25} (with $\eta\defi\mu$) which, combined to~\eqref{eq:GradHSo}, yields $\|\mu^{\frac12}\hTd\|_{0,\Omega}\lesssim\mu_\sharp^{\frac12}|\uhd|_{\Curl,\d}(\lesssim\mu_\sharp^{\frac12}\|\vec{j}\|_{0,\Omega})$.
  Since $|\uhd|_{\Curl,\d}=0$, we infer that $\uhd=\underline{\vec{0}}_\d$, thereby concluding the proof.
\end{proof}
\noindent
Remark that the validity of the discrete orthogonality condition~\eqref{eq:stabi}, on which crucially hinges the stability of the method, requires assumption~\eqref{eq:mu.d} on the magnetic permeability to be fulfilled.
Last, we prove an error estimate in energy-norm for the solution to Problem~\eqref{eq:field.dis}.
\begin{theorem}[Energy-error estimate] \label{th:field.esti}
  Assume that $\vec{h}\in\Hcurl$ solution to Problem~\eqref{eq:field.wea} further satisfies $\vec{h}\in\vec{H}^s(\t)$ for some $s\in(1,k+1]$, and $\Curl\vec{h}\in\vec{H}^r(\t)$ for some $r\in(s-1,k]$. Then, the following estimate holds true:
  \begin{equation} \label{eq:field.esti}
    |\uhd-\intHcurl(\vec{h})|_{\Curl,\d}^2+\sum_{T\in\t}\mu_T^2|\upT|_{\Grad,T}^2\lesssim\sum_{T\in\t}\big(h_T^{2r}|\Curl\vec{h}|_{r,T}^2+h_T^{2(s-1)}|\vec{h}|_{s,T}^2\big),
  \end{equation}
  where $(\uhd,\upd)\in\dHcurl\times\hdHSo$ is the unique solution to Problem~\eqref{eq:field.dis}.
\end{theorem}
\begin{proof}
  First, by Sobolev embedding, remark that since $\vec{h}\in\vec{H}^s(\t)$ for some $s>1$ ($s>0$ would be enough), then $\vec{h}\in\vec{L}^q(\Omega)$ for some $q>2$, and one can thus give a meaning to $\intHcurl(\vec{h})\in\dHcurl$. Let us set $|(\uwd,\urd)|_\d^2\defi A_\d(\uwd,\uwd)+S_\d(\urd,\urd)$.
  The starting point of the proof is the following simple inequality: for all $(\uwd,\urd)\in\dHcurl\times\hdHSo$,
  \[|(\uwd,\urd)|_\d\leq\max_{\overset{(\uvd,\uqd)\in\dHcurl\times\hdHSo,}{|(\uvd,\uqd)|_\d=1}}A_\d(\uwd,\uvd)+B_\d(\uvd,\urd)-B_\d(\uwd,\uqd)+S_\d(\urd,\uqd).\] 
  Letting then $\uwd=\big(\uhd-\intHcurl(\vec{h})\big)\in\dHcurl$ and $\urd=\big(\upd-\intHo(p)\big)=\upd\in\hdHSo$ (it is recalled that $p\equiv 0$), developping the expression into the max, and using Problem~\eqref{eq:field.dis}, yields
  \[\left(|\uhd-\intHcurl(\vec{h})|_{\Curl,\d}^2+\sum_{T\in\t}\mu_T^2|\upT|_{\Grad,T}^2\right)^{\nicefrac12}\lesssim\max_{\overset{(\uvd,\uqd)\in\dHcurl\times\hdHSo,}{|(\uvd,\uqd)|_\d=1}}{\cal E}_\d\big((\uvd,\uqd)\big),\]
  where we have also utilized the estimate~\eqref{eq:field.wp} in the left-hand side, and set ${\cal E}_\d\big((\uvd,\uqd)\big)\defi\big(\vec{j},\CTd(\uvd)\big)_\Omega-A_\d\big(\intHcurl(\vec{h}),\uvd\big)+B_\d\big(\intHcurl(\vec{h}),\uqd\big)$.
  Let us define
  \begin{align*}
    \mathfrak{T}_1\defi&\big(\vec{j},\CTd(\uvd)\big)_\Omega-A_\d\big(\intHcurl(\vec{h}),\uvd\big)\\=&\big(\Curl\vec{h}-\big(\CTd\circ\intHcurl\big)(\vec{h}),\CTd(\uvd)\big)_{\Omega}-S_{\Curl,\d}\big(\intHcurl(\vec{h}),\uvd\big),
  \end{align*}
  and
  \[\mathfrak{T}_2\defi B_\d\big(\intHcurl(\vec{h}),\uqd\big)=\big(\bProjTd(\mu\vec{h}),\GTd(\uqd)\big)_{\Omega},\]
  where we have invoked~\eqref{eq:field.str.a} to replace $\vec{j}$ by $\Curl\vec{h}$ in $\mathfrak{T}_1$, and used the fact that $\mu$ is piecewise constant in $\mathfrak{T}_2$.
  For $\mathfrak{T}_1$, using (i) the local commutation property~\eqref{eq:dHcurlT.comm} along with the fact that $\Curl\vec{h}\in\vec{H}^r(\t)$, and (ii) the fact that $\vec{h}\in\vec{H}^s(\t)$ for some $s>1$ ($s>\frac12$ would be enough), so that the full trace of $\vec{h}$ can be given a meaning in $\LL[F]^3$ on each interface (on both sides) and boundary face, we infer from the fractional approximation results of~\cite[Lem.~2.5]{CEPig:21} that
  \[\mathfrak{T}_1\lesssim\left(\sum_{T\in\t}\big(h_T^{2r}|\Curl\vec{h}|_{r,T}^2+h_T^{2(s-1)}|\vec{h}|_{s,T}^2\big)\right)^{\nicefrac12}A_\d(\uvd,\uvd)^{\nicefrac12}.\]
  For $\mathfrak{T}_2$ now, using the definition~\eqref{eq:GT} of the (local) gradient reconstruction, and integrating by parts the volumetric term therein, we infer that
  \[\mathfrak{T}_2=\big(\mu\vec{h},\Grad_\t\!\qTd\big)_{\Omega}+\sum_{T\in\t}\sum_{F\in\f_T}\big(\bProjT(\mu_T\vec{h}_{\mid T}\!)_{\mid F}{\cdot}\normal_{T,F},q_{F,T}-q_{T\mid F}\big)_F,\]
  where we have invoked the fact that $\Grad_\t\!\qTd\in\bPoly^{k-2}(\t)$ to remove the projector in the first term.
  Integrating by parts (cell by cell) the first term in the right-hand side, and combining $\vec{h}\in\Hzdivz$ (owing to~\eqref{eq:field.str.b} and~\eqref{eq:field.str.c}) and $\uqd\in\hdHSo$ (cf.~\eqref{eq:hdHSo} and~\eqref{eq:hdHSof}), yields
  \ifSISC
  \begin{multline*}
    \mathfrak{T}_2=\sum_{T\in\t}\sum_{F\in\f_T}\varepsilon_{T,F}\big(\big(\bProjT(\mu_T\vec{h}_{\mid T}\!)_{\mid F}-(\mu\vec{h})_{\mid F}\big){\cdot}\normal_{F},q_{F,T}-q_{T\mid F}\big)_F\\+\sum_{i=1}^{\beta_1}\sum_{F\in\f^\circ_{\Sigma_i}}\sigma_i\big((\mu\vec{h})_{\mid F}{\cdot}\normal_F,1\big)_F,
  \end{multline*}
  \else
  \[\mathfrak{T}_2=\sum_{T\in\t}\sum_{F\in\f_T}\varepsilon_{T,F}\big(\big(\bProjT(\mu_T\vec{h}_{\mid T})_{\mid F}-(\mu\vec{h})_{\mid F}\big){\cdot}\normal_{F},q_{F,T}-q_{T\mid F}\big)_F+\sum_{i=1}^{\beta_1}\sum_{F\in\f^\circ_{\Sigma_i}}\sigma_i\big((\mu\vec{h})_{\mid F}{\cdot}\normal_F,1\big)_F,\]
  \fi
  where we have additionally leveraged that $\mu\vec{h}\in\vec{H}^s(\t)$ for some $s>1$ ($s>\frac{1}{2}$ would be sufficient) to give a standard meaning to its normal traces.
  \ifSISC
  Now, since, for any $i\in\{1,\ldots,\beta_1\}$, there holds $\sum_{F\in\f^\circ_{\Sigma_i}}\sigma_i\big((\mu\vec{h})_{\mid F}{\cdot}\normal_F,1\big)_F=\sigma_i\big((\mu\vec{h})_{\mid\Sigma_i}{\cdot}\normal_{\Sigma_i},1\big)_{\Sigma_i}$, we infer by~\eqref{eq:field.str.d} that $\mathfrak{T}_2=\sum_{T\in\t}\sum_{F\in\f_T}\varepsilon_{T,F}\big(\big(\bProjT(\mu_T\vec{h}_{\mid T}\!)_{\mid F}-(\mu\vec{h})_{\mid F}\big){\cdot}\normal_{F},q_{F,T}-q_{T\mid F}\big)_F$.
  \else
  Now, since, for any $i\in\{1,\ldots,\beta_1\}$, there holds $\sum_{F\in\f^\circ_{\Sigma_i}}\sigma_i\big((\mu\vec{h})_{\mid F}{\cdot}\normal_F,1\big)_F=\sigma_i\big((\mu\vec{h})_{\mid\Sigma_i}{\cdot}\normal_{\Sigma_i},1\big)_{\Sigma_i}$, we infer by~\eqref{eq:field.str.d} that
  \begin{equation*}
    \mathfrak{T}_2=\sum_{T\in\t}\sum_{F\in\f_T}\varepsilon_{T,F}\big(\big(\bProjT(\mu_T\vec{h}_{\mid T}\!)_{\mid F}-(\mu\vec{h})_{\mid F}\big){\cdot}\normal_{F},q_{F,T}-q_{T\mid F}\big)_F.
  \end{equation*}
  \fi
  The approximation results of~\cite[Lem.~2.5]{CEPig:21}, along with standard arguments, then show that
  \[\mathfrak{T}_2\lesssim\left(\sum_{T\in\t}h_T^{2(s-1)}|\vec{h}|_{s,T}^2\right)^{\nicefrac12}\left(\sum_{T\in\t}\mu_T^2S_{\Grad,T}(\uqT,\uqT)\right)^{\nicefrac12}.\]
  Since $\sum_{T\in\t}\mu_T^2S_{\Grad,T}(\uqT,\uqT)\leq S_\d(\uqd,\uqd)$, gathering the estimates on $\mathfrak{T}_1$ and $\mathfrak{T}_2$, and using that $|(\uvd,\uqd)|_\d=1$, we conclude the proof.
\end{proof}
\noindent
From the error estimate~\eqref{eq:field.esti}, we infer that the HHO scheme converges with optimal rate $h_\d^k$ when $\vec{h}\in\vec{H}^{k+1}(\t)$.
Before concluding this section, a few remarks are in order.
\begin{remark}[Low-regularity solutions and error measure] \label{re:reg}
  In order to yield convergence, the energy-error estimate from Theorem~\ref{th:field.esti} essentially requires that $\vec{h}\in\vec{H}^s(\t)$ for some $s>1$.
  By Remark~\ref{re:reg.field}, this regularity assumption on the solution is not met in general, even for a globally constant coefficient $\mu$ on a convex domain.
  Recovering convergence for any $s>0$ and $r\geq 0$ would require to consider a weaker error measure, morally on $h_\t\CTd$ instead of $\CTd$.
  This is in line with the fact that, for first-order models, the only meaningful error measure is the $L^2$-error on $\vec{h}$, which formally scales like $h_\t\CTd$.
  For $s\in(0,\frac12]$, the analysis would also necessitate to give a meaning to the face terms (adapting, for example, the arguments from~\cite[Sec.~5]{ErnGu:22}).
  We do not pursue herein further in this direction.
\end{remark}
\begin{remark}[Variant of the scheme] \label{re:variant}
  Since the model at hand is first-order, following~\cite[Sec.~3.1]{CDPLe:22}, it is also possible to define the discrete rotational operator as the broken rotational $\Curl_\t$, instead of $\CTd$.
  In practice, the new HHO scheme is obtained replacing in Problem~\eqref{eq:field.dis} every occurrence of $\CTd$ by $\Curl_\t$.
  The main advantage of doing so is that one can then relax the assumption~\eqref{eq:pQF.bis}, and consider in place of~\eqref{eq:pQF.choice} the leaner space $\pQF{k}\defi\pRF{k}$.
  This choice is actually sufficient (and necessary, thus optimal) for ensuring stability (cf.~\cite[Rmk.~10]{LePit:25}).
  Redefining $A_T(\uwT,\uvT)$, there now holds (in place of~\eqref{eq:field.AT}) that $A_T(\uvT,\uvT)=|\uvT|_{\Curl,T}^2$, and this identity holds true irrespectively of~\eqref{eq:pQF.bis}.
  There is a price to pay for this complexity reduction. The commutation property~\eqref{eq:dHcurlT.comm} is no longer valid in that case, and the proxy for $\Curl\vec{h}$ is then given by $\Curl_\t\!\big(\bProjTd(\vec{h})\big)$ in lieu of $\bProjTd[k-1](\Curl\vec{h})$. As supported by our numerical experiments (cf.~Section~\ref{sse:singular}), for singular solutions, this alternative proxy might unfortunately retain insufficient approximability properties.
\end{remark}
\begin{remark}[About stability] \label{re:tet}
  Assume $\t$ is a member of a tetrahedral mesh sequence. Then, as already pointed out in~\cite[Rmk.~3.1]{CDPLe:22}, the bilinear contribution $S_\d(\upd,\uqd)$ can be removed from Problem~\eqref{eq:field.dis} without compromising its stability.
  In the tetrahedral case,~\eqref{eq:GT.coer} indeed yields $\sum_{T\in\t}\mu_T^2|\uqT|_{\Grad,T}^2\lesssim\sum_{T\in\t}\mu_T^2h_T^2\|\GT(\uqT)\|_{0,T}^2$ for all $\uqd\in\hdHo$.
  Since the discrete saddle-point system enforces a control on the latter right-hand side (cf.~\cite[Lem.~3.6]{CDPLe:22}), this actually grants stability.
  When $\t$ belongs to a general mesh sequence,~\eqref{eq:GT.equiv} indicates that $\sum_{T\in\t}\mu_T^2S_{\Grad,T}(\uqT,\uqT)$ needs also be controlled to ensure stability.
  This motivates the adding of the bilinear term $S_\d(\upd,\uqd)$, embedding $\sum_{T\in\t}\mu_T^2S_{\Grad,T}(\upT,\uqT)$, to Problem~\eqref{eq:field.dis}.
  Remarkably, the continuous Lagrange multiplier being zero, any choice of $S_\d(\upd,\uqd)$ is consistent.
  We here define $S_\d$ from~\eqref{eq:field.dbf.N}, which also embeds a volumetric contribution, with the idea of improving stability for the discrete problem.
  This strategy may yield slightly more accurate results in some situations (cf.~\cite[Rmk.~3.2]{CDPLe:22}).
  Other choices of the same nature can be made for $S_\d$, see e.g.~\cite[Eq.~(3.4c)]{CDPLe:22}.
  The present remark extends to Problem~\eqref{eq:vecpot.dis} (vector potential case).
\end{remark}

\subsubsection{Vector potential formulation} \label{ssse:vecpot.d}

Recall the definition~\eqref{eq:dHzcurl} of $\dHzcurl$, hybrid counterpart of $\Hzcurl$, and assume that the space $\pQF{k}$ is still given by~\eqref{eq:pQF.choice} for all $F\in\f$.
Recall, in addition, the definition~\eqref{eq:dHGo} of $\dHGo$, hybrid counterpart of $\HGo$.

Given $T\in\t$, we introduce the local bilinear forms $A_T:\dHcurlT\times\dHcurlT\to\Real$, $B_T:\dHcurlT\times\dHoT\to\Real$, and $S_T:\dHoT\times\dHoT\to\Real$ defined by: for all $\uvT,\uwT\in\dHcurlT$, and $\uqT,\urT\in\dHoT$,
\begin{subequations}
  \label{eq:vecpot.dbf}
  \begin{alignat}{1}
    A_T(\uwT,\uvT)&\defi\mu_T^{-1}\big(\CT(\uwT),\CT(\uvT)\big)_T+\mu_T^{-1}S_{\Curl,T}(\uwT,\uvT),\label{eq:vecpot.dbf.A}
    \\
    B_T(\uwT,\uqT)&\defi\big(\wT,\GT(\uqT)\big)_T,\label{eq:vecpot.dbf.B}
    \\
    S_T(\urT,\uqT)&\defi \mu_Th_T^2\big(\Grad\rT,\Grad\qT\big)_T+\mu_TS_{\Grad,T}(\urT,\uqT),\label{eq:vecpot.dbf.N}
  \end{alignat}
\end{subequations}
where the local stabilizers $S_{\Curl,T}$ and $S_{\Grad,T}$ are respectively defined by~\eqref{eq:dHcurlT.stab} and~\eqref{eq:dHoT.stab}.
Here as well, the scaling $\mu_T$ (also $h_T^2$) in~\eqref{eq:vecpot.dbf.N} is chosen so as to restore consistency between the physical units.
In turn, the global bilinear forms $A_\d:\dHcurl\times\dHcurl\to\Real$, $B_\d:\dHcurl\times\dHo\to\Real$, and $S_\d:\dHo\times\dHo\to\Real$ are assembled by summing the local contributions.
The discrete HHO problem reads: Find $(\uad,\upd)\in\dHzcurl\times\dHGo$ s.t.
\begin{subequations}\label{eq:vecpot.dis}
  \begin{alignat}{2}
    A_\d(\uad,\uvd) + B_\d(\uvd,\upd) & = \big(\vec{j},\vTd\big)_\Omega &\qquad&\forall \uvd\in\dHzcurl,\label{eq:vecpot.dis.a}
    \\
    -B_\d(\uad,\uqd) + S_\d(\upd,\uqd) & = 0 &\qquad&\forall\uqd\in\dHGo.\label{eq:vecpot.dis.b}
  \end{alignat}
\end{subequations}
Let us establish well-posedness for Problem~\eqref{eq:vecpot.dis}.
\begin{lemma}[Well-posedness] \label{le:vecpot.wp}
  For all $\uvd\in\dHcurl$ and $\uqd\in\dHo$, there holds:
  \begin{equation} \label{eq:vecpot.wp}
    \sum_{T\in\t}\mu_T^{-1}|\uvT|_{\Curl,T}^2+\sum_{T\in\t}\mu_T|\uqT|_{\Grad,T}^2\lesssim A_\d(\uvd,\uvd)+S_\d(\uqd,\uqd).
  \end{equation}
  As a consequence, Problem~\eqref{eq:vecpot.dis} is well-posed.
\end{lemma}
\begin{proof}
  Let $\uvT\in\dHcurlT$ and $\uqT\in\dHoT$.
  Starting from definition~\eqref{eq:vecpot.dbf.A}, and invoking~\eqref{eq:CT.equiv} (which is valid under assumption~\eqref{eq:pQF.choice}), there holds
  \begin{equation} \label{eq:vecpot.AT}
    A_T(\uvT,\uvT)=\mu_T^{-1}\|\CT(\uvT)\|_{0,T}^2+\mu_T^{-1}S_{\Curl,T}(\uvT,\uvT)\eqsim\mu_T^{-1}|\uvT|_{\Curl,T}^2.
  \end{equation}
  Similarly, by~\eqref{eq:vecpot.dbf.N}, $S_T(\uqT,\uqT)=\mu_T|\uqT|_{\Grad,T}^2$.
  Summing over $T\in\t$ yields~\eqref{eq:vecpot.wp}.

  Now, since the linear system corresponding to Problem~\eqref{eq:vecpot.dis} is square, proving well-posedness is equivalent to proving the uniqueness of its solution. Let us show that, if $\vec{j}\equiv\vec{0}$, then necessarily $(\uad,\upd)=\big(\underline{\vec{0}}_\d,\underline{0}_\d\big)$. First, remark that the map $|{\cdot}|_{\Grad,\d}$ (cf.~\eqref{eq:dHoT.sn}) is a norm on $\dHGo$ (cf.~\eqref{eq:dHzo} and~\eqref{eq:dHGo}), as a by-product of the discrete zero boundary condition on $\Gamma_0$. Now, let us test Problem~\eqref{eq:vecpot.dis} (for $\vec{j}\equiv\vec{0}$) with $\uvd=\uad$ and $\uqd=\upd$. Summing~\eqref{eq:vecpot.dis.a} and~\eqref{eq:vecpot.dis.b}, and leveraging~\eqref{eq:vecpot.wp}, we directly infer that $|\uad|_{\Curl,\d}=0$ and $|\upd|_{\Grad,\d}=0$.
  From the second, since $\upd\in\dHGo$, we deduce that $\upd=\underline{0}_\d$. Plugging this value into~\eqref{eq:vecpot.dis.b}, we get $B_\d(\uad,\uqd)=\big(\aTd,\GTd(\uqd)\big)_{\Omega}=0$ for all $\uqd\in\dHGo$.
  Since, for any $q\in\HGo$, $\intHo(q)\in\dHGo$, letting $\uqd=\intHo(q)$ in the latter relation, and leveraging the local commutation property~\eqref{eq:dHoT.comm}, along with $\aTd\in\bPoly^k(\t)$, we infer $\big(\aTd,\vec{z}\big)_{\Omega}=0$ for all $\vec{z}\in\Grad\big(\HGo\big)$.
  The conclusion then follows (recall that $\uad\in\dHzcurl$) from the first hybrid Weber inequality of~\cite[Rmk.~14]{LePit:25} (with $\eta\defi 1$) which, combined to~\eqref{eq:GradHGo}, yields $\|\aTd\|_{0,\Omega}\lesssim|\uad|_{\Curl,\d}(\lesssim\mu_\sharp\|\vec{j}\|_{0,\Omega})$. Since $|\uad|_{\Curl,\d}=0$, we infer $\uad=\underline{\vec{0}}_\d$.
\end{proof}
\noindent
Remark that, contrary to what held true for the field formulation, the stability of Problem~\eqref{eq:vecpot.dis} is not conditional to assumption~\eqref{eq:mu.d}.
In order to simplify the analysis below, we nonetheless stick to~\eqref{eq:mu.d}, and refer to Section~\ref{sse:variable} for a numerical assessment in the case of locally variable permeability.
We prove an error estimate in energy-norm for the solution to Problem~\eqref{eq:vecpot.dis}.
\begin{theorem}[Energy-error estimate] \label{th:vecpot.esti}
  Assume that the solution $\vec{a}\in\Hzcurl$ to Problem~\eqref{eq:vecpot.wea} further satisfies $\vec{a}\in\vec{H}^s(\t)$ for some $s\in(1,k+1]$, and $\Curl\vec{a}\in\vec{H}^r(\t)$ for some $r\in(\max(\frac{1}{2},s-1),k]$. Then, the following estimate holds true:
  \begin{equation} \label{eq:vecpot.esti}
    \sum_{T\in\t}\mu_T^{-1}|\uaT-\intHcurlT(\vec{a}_{\mid T})|_{\Curl,T}^2+\sum_{T\in\t}\mu_T|\upT|_{\Grad,T}^2\lesssim\sum_{T\in\t}\mu_T^{-1}\big(h_T^{2r}|\Curl\vec{a}|_{r,T}^2+h_T^{2(s-1)}|\vec{a}|_{s,T}^2\big),
  \end{equation}
  where $(\uad,\upd)\in\dHzcurl\times\dHGo$ is the unique solution to Problem~\eqref{eq:vecpot.dis}.
\end{theorem}
\begin{proof}
  First, by Sobolev embedding, remark that since $\vec{a}\in\vec{H}^s(\t)$ for some $s>1$ ($s>0$ would be enough), then $\vec{a}\in\vec{L}^q(\Omega)$ for some $q>2$, and one can thus give a meaning to $\intHcurl(\vec{a})\in\dHzcurl$.
  Let us set $|(\uwd,\urd)|_\d^2\defi A_\d(\uwd,\uwd)+S_\d(\urd,\urd)$.
  Proceeding as in the proof of Theorem~\ref{th:field.esti}, we infer that
  \[\left(\sum_{T\in\t}\mu_T^{-1}|\uaT-\intHcurlT(\vec{a}_{\mid T})|_{\Curl,T}^2+\sum_{T\in\t}\mu_T|\upT|_{\Grad,T}^2\right)^{\nicefrac12}\lesssim\max_{\overset{(\uvd,\uqd)\in\dHzcurl\times\dHGo,}{|(\uvd,\uqd)|_\d=1}}{\cal E}_\d\big((\uvd,\uqd)\big),\]
  where ${\cal E}_\d\big((\uvd,\uqd)\big)\defi\big(\vec{j},\vTd\big)_\Omega-A_\d\big(\intHcurl(\vec{a}),\uvd\big)+B_\d\big(\intHcurl(\vec{a}),\uqd\big)$.
  Let us define
  \begin{equation*}
    \mathfrak{T}_1\defi\big(\vec{j},\vTd\big)_\Omega-A_\d\big(\intHcurl(\vec{a}),\uvd\big),\quad\mathfrak{T}_2\defi B_\d\big(\intHcurl(\vec{a}),\uqd\big)=\big(\bProjTd(\vec{a}),\GTd(\uqd)\big)_{\Omega}.
  \end{equation*}
  For $\mathfrak{T}_1$, (i) replacing $\vec{j}$ by $\Curl(\mu^{-1}\Curl\vec{a})$ (leveraging~\eqref{eq:vecpot.str.a}), then performing a cell-by-cell integration by parts in $\big(\vec{j},\vTd\big)_\Omega$, and (ii) using the definition~\eqref{eq:CT} of the (local) rotational reconstruction operator for $A_\d\big(\intHcurl(\vec{a}),\uvd\big)$, we infer that
  \begin{multline*}
    \mathfrak{T}_1=\big(\mu^{-1}\big(\Curl\vec{a}-\big(\CTd\circ\intHcurl\big)(\vec{a})\big),\Curl_\t\!\vTd\big)_{\Omega}-\sum_{T\in\t}\mu_T^{-1}S_{\Curl,T}\big(\intHcurlT(\vec{a}_{\mid T}),\uvT\big)\\
    +\sum_{T\in\t}\sum_{F\in\f_T}\varepsilon_{T,F}\big(\vec{v}_{T\mid F}{\times}\normal_F-\vFt,\normal_F{\times}\big(\mu^{-1}\Curl\vec{a}-\mu_T^{-1}\big(\CT\circ\intHcurlT\big)(\vec{a}_{\mid T})\big)_{\mid F}{\times}\normal_F\big)_F,
  \end{multline*}
  where we have also utilized the fact that $\mu^{-1}\Curl\vec{a}\in\Hcurl$ with $\mu^{-1}\Curl\vec{a}\in\vec{H}^r(\t)$ for some $r>\frac{1}{2}$, together with the fact that $\uvd\in\dHzcurl$.
  Let us first estimate the two first addends in $\mathfrak{T}_1$, whose sum we denote by $\mathfrak{T}_{1,1}$ ($\mathfrak{T}_{1,2}$ will denote the last addend of $\mathfrak{T}_1$).
  By (i) the local commutation property~\eqref{eq:dHcurlT.comm}, along with the fact that $\Curl\vec{a}\in\vec{H}^r(\t)$, and (ii) the fact that $\vec{a}\in\vec{H}^s(\t)$ for some $s>1$ ($s>\frac{1}{2}$ would be enough to give a meaning to the boundary terms), we get from the fractional approximation results of~\cite[Lem.~2.5]{CEPig:21} that
  \[\mathfrak{T}_{1,1}\lesssim\left(\sum_{T\in\t}\mu_T^{-1}\big(h_T^{2r}|\Curl\vec{a}|_{r,T}^2+h_T^{2(s-1)}|\vec{a}|_{s,T}^2\big)\right)^{\nicefrac12}\left(\sum_{T\in\t}\mu_T^{-1}|\uvT|_{\Curl,T}^2\right)^{\nicefrac12}.\]
  In turn, using~\cite[Lem.~9]{LePit:25} (yielding full control of tangential jumps), together with similar arguments, we infer that
  \[\mathfrak{T}_{1,2}\lesssim\left(\sum_{T\in\t}\mu_T^{-1}h_T^{2r}|\Curl\vec{a}|_{r,T}^2\right)^{\nicefrac12}\left(\sum_{T\in\t}\mu_T^{-1}|\uvT|_{\Curl,T}^2\right)^{\nicefrac12}.\]
  By~\eqref{eq:vecpot.AT}, we finally get $\mathfrak{T}_1\lesssim\left(\sum_{T\in\t}\mu_T^{-1}\big(h_T^{2r}|\Curl\vec{a}|_{r,T}^2+h_T^{2(s-1)}|\vec{a}|_{s,T}^2\big)\right)^{\nicefrac12}A_\d(\uvd,\uvd)^{\nicefrac12}$.
  Let us now estimate $\mathfrak{T}_2$. Using the definition~\eqref{eq:GT} of the (local) gradient reconstruction (with the fact that $\uqd\in\dHGo$), and integrating by parts the volumetric term therein, we infer
  \[\mathfrak{T}_2=\big(\vec{a},\Grad_\t\!\qTd\big)_{\Omega}+\sum_{T\in\t}\sum_{F\in\f_T}\big(\bProjT(\vec{a}_{\mid T})_{\mid F}{\cdot}\normal_{T,F},\qF-q_{T\mid F}\big)_F.\]
  A cell-by-cell integration by parts of the first term in the right-hand side, and the combination of $\vec{a}\in\vec{H}(\Div^0;\Omega)$ (according to~\eqref{eq:vecpot.str.b}) and $\uqd\in\dHGo$ (cf.~\eqref{eq:dHzo} and~\eqref{eq:dHGo}), yields
  \[\mathfrak{T}_2=\sum_{T\in\t}\sum_{F\in\f_T}\varepsilon_{T,F}\big(\big(\bProjT(\vec{a}_{\mid T})_{\mid F}-\vec{a}_{\mid F}\big){\cdot}\normal_{F},\qF-q_{T\mid F}\big)_F+\sum_{j=1}^{\beta_2}\sum_{F\in\f^\partial_{\Gamma_j}}\gamma_j\big(\vec{a}_{\mid F}{\cdot}\normal_F,1\big)_F,\]
  where we have additionally utilized the fact that $\vec{a}\in\vec{H}^s(\t)$ for some $s>1$ ($s>\frac{1}{2}$ would be sufficient) to give a standard meaning to its normal traces.
  \ifSISC
  Now, since, for any $j\in\{1,\ldots,\beta_2\}$, there holds $\sum_{F\in\f^\partial_{\Gamma_j}}\gamma_j\big(\vec{a}_{\mid F}{\cdot}\normal_F,1\big)_F=\gamma_j\big(\vec{a}_{\mid\Gamma_j}{\cdot}\normal,1\big)_{\Gamma_j}$, we infer by~\eqref{eq:vecpot.str.d} that $\mathfrak{T}_2=\sum_{T\in\t}\sum_{F\in\f_T}\varepsilon_{T,F}\big(\big(\bProjT(\vec{a}_{\mid T})_{\mid F}-\vec{a}_{\mid F}\big){\cdot}\normal_{F},\qF-q_{T\mid F}\big)_F$.
  \else
  Now, since, for any $j\in\{1,\ldots,\beta_2\}$, there holds $\sum_{F\in\f^\partial_{\Gamma_j}}\gamma_j\big(\vec{a}_{\mid F}{\cdot}\normal_F,1\big)_F=\gamma_j\big(\vec{a}_{\mid\Gamma_j}{\cdot}\normal,1\big)_{\Gamma_j}$, we infer by~\eqref{eq:vecpot.str.d} that
  \begin{equation*}
    \mathfrak{T}_2=\sum_{T\in\t}\sum_{F\in\f_T}\varepsilon_{T,F}\big(\big(\bProjT(\vec{a}_{\mid T})_{\mid F}-\vec{a}_{\mid F}\big){\cdot}\normal_{F},\qF-q_{T\mid F}\big)_F.
  \end{equation*}
  \fi
  The approximation results of~\cite[Lem.~2.5]{CEPig:21}, along with standard arguments, then show that
  \[\mathfrak{T}_2\lesssim\left(\sum_{T\in\t}\mu_T^{-1}h_T^{2(s-1)}|\vec{a}|_{s,T}^2\right)^{\nicefrac12}\left(\sum_{T\in\t}\mu_TS_{\Grad,T}(\uqT,\uqT)\right)^{\nicefrac12}.\]
  Since $\sum_{T\in\t}\mu_TS_{\Grad,T}(\uqT,\uqT)\leq S_\d(\uqd,\uqd)$, gathering the estimates on $\mathfrak{T}_1$ and $\mathfrak{T}_2$, and using that $|(\uvd,\uqd)|_\d=1$, we conclude the proof.
\end{proof}
\noindent
From the error estimate~\eqref{eq:vecpot.esti}, we infer that the HHO scheme converges with optimal rate $h_\d^k$ when $\vec{a}\in\vec{H}^{k+1}(\t)$.
\ifSISC
\else
A few remarks are necessary before concluding this section.
\fi
\begin{remark}[Low-regularity solutions]
  In order to yield convergence, the energy-error estimate from Theorem~\ref{th:vecpot.esti} essentially requires that $\vec{a}\in\vec{H}^s(\t)$ for some $s>1$.
  By Remark~\ref{re:reg.vecpot}, this regularity assumption on the solution is not met in general, even on a convex domain.
  However, it is always true that $\vec{a}\in\vec{H}^s(\t)$ for some $s>\frac12$.
  Recovering convergence for any $s>\frac12$ would then necessitate to consider a weaker error measure for the scheme, formally on $h_\t^{\nicefrac12}\CTd$.
  This is precisely the approach pursued in~\cite{ErnGu:23,ErnGu:25} in the DG context.
  Note that the correct treatment of interface terms also requires that $\mu^{-1}\Curl\vec{a}\in\vec{H}^r(\t)$ for some $r>\frac12$.
  By Remark~\ref{re:reg.vecpot}, the latter regularity assumption is met for a globally constant coefficient $\mu$, but is not satisfied in general.
  Adapting the arguments from~\cite[Sec.~5]{ErnGu:22}, it should however be possible to relax this assumption to $r>0$, which would cover the general case.
  We do not pursue herein further in this direction.
\end{remark}
\ifSISC
\else
\begin{remark}[Variant of the scheme]
  It is possible, instead of (locally) reconstructing the rotational in $\bPoly^{k-1}(T)$ as in~\eqref{eq:CT}, to reconstruct it in the smaller space $\pRT{k-1}=\Curl\big(\bPoly^k(T)\big)$.
  This was the strategy advocated in~\cite[Sec.~3.2]{CDPLe:22}.
  In this case, the proxy for $\Curl\vec{a}$ is given by $\vec{\pi}^{k-1}_{\vec{\mathcal{R}},\t}(\Curl\vec{a})$ (with obvious notation for the projector) in place of $\bProjTd[k-1](\Curl\vec{a})$.
  This proxy has approximability properties no better than $\Curl_\t\!\big(\bProjTd(\vec{a})\big)$ (see Remark~\ref{re:variant}).
  Therefore, since this alternative construction does not allow to relax~\eqref{eq:pQF.bis}, and given that the complexity gain in using $\pRT{k-1}$ instead of $\bPoly^{k-1}(T)$ to reconstruct the rotational is marginal (local systems are small, and local computations are embarrassingly parallel), it is definitely a better option to stick to the choice $\bPoly^{k-1}(T)$ as in~\eqref{eq:CT}.
\end{remark}
\begin{remark}[Robustness to heterogeneity]
  Let us introduce $\kappa_\mu\defi\mu_\sharp/\mu_\flat\geq 1$ the (global) heterogeneity ratio of the magnetic permeability $\mu$.
  From~\eqref{eq:vecpot.esti}, one can infer that $|\underline{\vec{a}}_\d-\intHcurl(\vec{a})|_{\Curl,\d}$ is bounded by $\kappa_\mu^{\nicefrac12}$. Since $\vec{b}=\Curl\vec{a}$ (recall that $\vec{a}$ is the magnetic vector potential), notice that $|\underline{\vec{a}}_\d-\intHcurl(\vec{a})|_{\Curl,\d}$ is actually an $L^2$-like measure of the scheme error on $\vec{b}$. The computed magnetic induction is thus only mildly sensitive to heterogeneity.
\end{remark}
\fi

\section{Numerical experiments} \label{se:num}

We perform in this section a comprehensive numerical assessment of the HHO schemes we devised and analyzed in Section~\ref{sse:hho}.
One important goal is to validate numerically the theoretical rates of convergence of the methods. Our focus is on relatively simple geometries, and on test-cases for which an analytical solution is available.
The computational domains under study include topologically non-trivial geometries, as well as a domain with reentrant edge, for which the corresponding solution is singular.

We consider two families of (regular) mesh sequences, respectively composed of \emph{tetrahedral} and \emph{polyhedral} meshes.
The polyhedral meshes are constructed from the tetrahedral ones by agglomerating tetrahedral elements into polyhedra in a random fashion.
The agglomeration procedure concerns all tetrahedra surrounding randomly selected mesh vertices, in such a way that one can construct a mesh sequence whose regularity parameter does not decrease upon refinement.
We here solve the discrete Problems~\eqref{eq:field.dis} and~\eqref{eq:vecpot.dis}, respectively referred to as first- and second-order Problems, modulo the following adaptations: (i) building on Remark~\ref{re:variant}, for the first-order Problem, when the exact solution is regular, we replace the discrete rotational operator $\CTd$  by its broken version $\Curl_\t$, and the magnetic face space is then set to $\pQF{k}\defi\pRF{k}$ for all $F\in\f$; building on Remark~\ref{re:tet}, (ii) when $\t$ is a tetrahedral mesh, we remove from both the first- and second-order Problems the contribution $S_\d(\upd,\uqd)$, whereas (iii) when $\t$ is a polyhedral mesh, we remove from both the first- and second-order Problems, in the definition of $S_\d(\upd,\uqd)$, the contribution stemming from the local bilinear form $h_T^2\big(\Grad\pT,\Grad\qT\big)_T$.
Doing so, we achieve, without compromising stability, a somewhat leaner implementation, and we bear consistency with the numerical results from~\cite{CDPLe:22}, where these variants were actually implemented (in the case of topologically trivial domains only).
Importantly, and as already noticed in~\cite{CDPLe:22}, the numerical tests that we performed (not reported here) show that these variants achieve virtually the same accuracy as the original schemes (and, interestingly, for examples with non-trivial topology, even slightly improve it).
As standard with skeletal methods, for each discrete Problem, both the magnetic and pressure-like cell unknowns are locally eliminated (without additional fill-in) from the global linear system using a Schur complement technique.
The resulting (condensed) global linear system is solved using the \texttt{SparseLU} direct solver of the \texttt{Eigen} library, on a laptop Dell Precision 5570 equipped with an Intel Core i7-12800H processor clocked at 2.80GHz, and with 64Gb of RAM.
The implementation uses the open-source C++ library ParaSkel++\footnote{cf.~\url{https://gitlab.inria.fr/simlemai/paraskel} (under GNU LGPL v3.0).}~\cite{BeaLe:21}.
For each computed discrete solution $\uud\in\{\uhd,\uad\}$, corresponding to the exact solution $\vec{u}\in\{\vec{h},\vec{a}\}$, we evaluate the relative \emph{energy-error} and \emph{$L^2$-error} by
\[\verb!En_err!\defi\frac{\left(\sum_{T\in\t}\eta_T^{-1}|\uuT-\intHcurlT(\vec{u}_{\mid T})|_{\Curl,T}^2\right)^{\nicefrac12}}{\left(\sum_{T\in\t}\eta_T^{-1}|\intHcurlT(\vec{u}_{\mid T})|_{\Curl,T}^2\right)^{\nicefrac12}}\quad\text{and}\quad\verb!L2_err!\defi\frac{\|\uTd-\bProjTd(\vec{u})\|_{0,\Omega}}{\|\bProjTd(\vec{u})\|_{0,\Omega}},\]
with $\eta\defi 1$ in the first-order case, and $\eta\defi\mu$ in the second-order one.
For all the numerical examples below, except in Section~\ref{sse:variable}, we set $\mu\equiv 1$.
Whenever the computational domain is non-simply-connected, and some additional flux constraints need to be enforced through cutting surfaces so as to guarantee uniqueness for the solution, one needs to compute, from the partition at hand, a discrete representative of each cutting surface. To this aim, we implemented a simplified version of the (co)homology computation algorithm proposed in~\cite{ARBGS:18}. We refer the reader to Figure~\ref{fig:cut} for an example, in the case of a torus discretized by a tetrahedral mesh, of a discrete cutting surface computed by the latter algorithm.

\begin{figure}[h!]
  \begin{center}
    \includegraphics[scale=0.22]{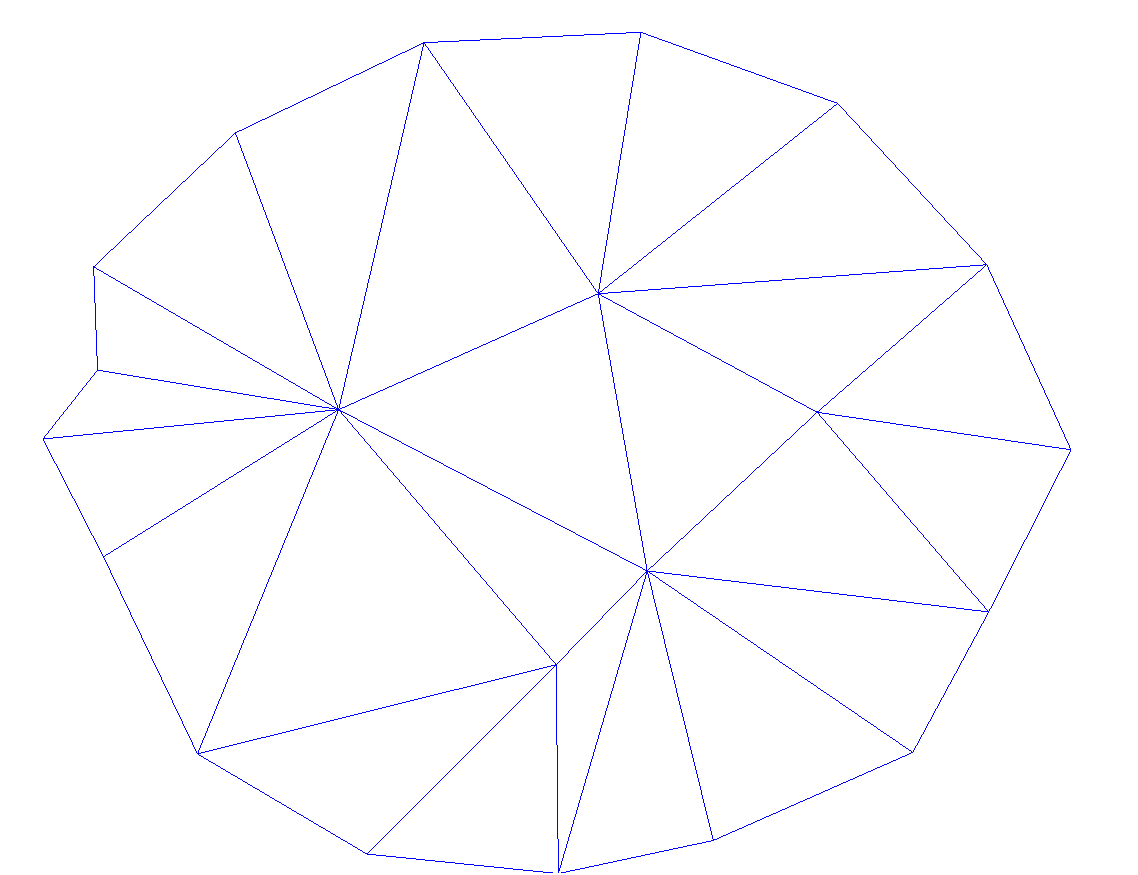}
  \end{center}
  \caption{\label{fig:cut}
    {\bf Test-case~\ref{sse:torus} (torus).} Example of a discrete cutting surface computed by the algorithm from~\cite{ARBGS:18}. The surface is piecewise planar, and the depicted triangles are faces of the tetrahedral mesh discretizing $\Omega$.}
\end{figure}

\subsection{Torus} \label{sse:torus}

Our first example consists in a toroidal domain $\Omega$, with midline radius $R=2$ and internal radius $r=1$.
For this geometry, $\beta_1=1$, and $\beta_2=0$.
We consider the first-order Problem~\eqref{eq:field.str}, with following smooth exact solution, expressed in cylindrical coordinates:
\begin{equation} \label{eq:sol.torus}
  \vec{h}(\rho,\varphi,z)\defi\Big(\cos\Big(\frac{\pi}{r}\sqrt{(\rho-R)^2+z^2}\Big)-c\Big)\hat{\vec{\varphi}},
\end{equation}
where the constant $c\in\Real$ is set so that, for some cutting surface $\Sigma$, the condition $\langle\vec{h}_{\mid\Sigma}{\cdot}\normal_{\Sigma},1\rangle_{\Sigma}=0$ is satisfied.
In our numerical experiments, the electric current density $\vec{j}$ is set according to~\eqref{eq:field.str.a}, whereas the zero normal boundary condition~\eqref{eq:field.str.c} and zero flux condition~\eqref{eq:field.str.d} (valid, respectively, on the curved boundary $\Gamma$, and on the $\Gamma$-delineated cutting surface $\Sigma$) are replaced by their non-homogeneous versions computed from~\eqref{eq:sol.torus} on the given piecewise polygonal approximations of $\Gamma$ and $\Sigma$.
For each $k\in\{1,2,3\}$, we solve Problem~\eqref{eq:field.dis} (with $\Curl_\t$), and depict on Figure~\ref{fig:torus}, for both the tetrahedral and (agglomerated) polyhedral mesh families, the relative energy-error and $L^2$-error as functions of the mesh size ($h_\d$), and of the number of degrees of freedom (\#DoF) after static condensation.
For the two mesh families, as predicted by Theorem~\ref{th:field.esti}, we obtain a convergence rate for the energy-error of order $k$.
We also observe a convergence rate of order $k+1$ for the $L^2$-error on the magnetic field.

\begin{figure}[h!]
  \centering
  \ref{legend:torus}
  \vspace{0.25cm}\\
  \begin{minipage}{0.40\textwidth}
    \begin{tikzpicture}[scale=0.63]
      \begin{loglogaxis}[legend columns=3, legend to name=legend:torus]
        \addplot[style=solid,color=blue,mark=*] table[x=meshsize,y=errXnorm_u] {dat/k1_torus_tet.dat};
        \addplot[style=solid,color=red,mark=square*] table[x=meshsize,y=errXnorm_u] {dat/k2_torus_tet.dat};
        \addplot[style=solid,color=brown,mark=diamond*] table[x=meshsize,y=errXnorm_u] {dat/k3_torus_tet.dat};
        \addplot[style=dashed,color=blue,mark=*,mark options={solid}] table[x=meshsize,y=errXnorm_u] {dat/k1_torus_pol.dat};
        \addplot[style=dashed,color=red,mark=square*,mark options={solid}] table[x=meshsize,y=errXnorm_u] {dat/k2_torus_pol.dat};
        \addplot[style=dashed,color=brown,mark=diamond*,mark options={solid}] table[x=meshsize,y=errXnorm_u] {dat/k3_torus_pol.dat};
        \logLogSlopeTriangle{0.90}{0.4}{0.1}{1}{black};
        \logLogSlopeTriangle{0.90}{0.4}{0.1}{2}{black};
        \logLogSlopeTriangle{0.90}{0.4}{0.1}{3}{black};
        \legend{{\small $k=1$ (tet)},{\small $k=2$ (tet)},{\small $k=3$ (tet)},{\small $k=1$ (pol)},{\small $k=2$ (pol)},{\small $k=3$ (pol)}};
      \end{loglogaxis}
    \end{tikzpicture}
  \end{minipage}
  \hspace{0.025\textwidth}
  \begin{minipage}{0.40\textwidth}
    \begin{tikzpicture}[scale=0.63]
      \begin{loglogaxis}
        \addplot[style=solid,color=blue,mark=*] table[x=n_DOFs,y=errXnorm_u] {dat/k1_torus_tet.dat};
        \addplot[style=solid,color=red,mark=square*] table[x=n_DOFs,y=errXnorm_u] {dat/k2_torus_tet.dat};
        \addplot[style=solid,color=brown,mark=diamond*] table[x=n_DOFs,y=errXnorm_u] {dat/k3_torus_tet.dat};
        \addplot[style=dashed,color=blue,mark=*,mark options={solid}] table[x=n_DOFs,y=errXnorm_u] {dat/k1_torus_pol.dat};
        \addplot[style=dashed,color=red,mark=square*,mark options={solid}] table[x=n_DOFs,y=errXnorm_u] {dat/k2_torus_pol.dat};
        \addplot[style=dashed,color=brown,mark=diamond*,mark options={solid}] table[x=n_DOFs,y=errXnorm_u] {dat/k3_torus_pol.dat};
        \logLogSlopeTriangleNDOFs{0.10}{-0.4}{0.1}{1/3}{black};
        \logLogSlopeTriangleNDOFs{0.10}{-0.4}{0.1}{2/3}{black};
        \logLogSlopeTriangleNDOFs{0.10}{-0.4}{0.1}{1}{black};
      \end{loglogaxis}
    \end{tikzpicture}
  \end{minipage}
  \vspace{0.2cm}\\
  \begin{minipage}{0.40\textwidth}
    \begin{tikzpicture}[scale=0.63]
      \begin{loglogaxis}
        \addplot[style=solid,color=blue,mark=*] table[x=meshsize,y=errL2_u] {dat/k1_torus_tet.dat};
        \addplot[style=solid,color=red,mark=square*] table[x=meshsize,y=errL2_u] {dat/k2_torus_tet.dat};
        \addplot[style=solid,color=brown,mark=diamond*] table[x=meshsize,y=errL2_u] {dat/k3_torus_tet.dat};
        \addplot[style=dashed,color=blue,mark=*,mark options={solid}] table[x=meshsize,y=errL2_u] {dat/k1_torus_pol.dat};
        \addplot[style=dashed,color=red,mark=square*,mark options={solid}] table[x=meshsize,y=errL2_u] {dat/k2_torus_pol.dat};
        \addplot[style=dashed,color=brown,mark=diamond*,mark options={solid}] table[x=meshsize,y=errL2_u] {dat/k3_torus_pol.dat};
        \logLogSlopeTriangle{0.90}{0.4}{0.1}{2}{black};
        \logLogSlopeTriangle{0.90}{0.4}{0.1}{3}{black};
        \logLogSlopeTriangle{0.90}{0.4}{0.1}{4}{black};
      \end{loglogaxis}
    \end{tikzpicture}
  \end{minipage}
  \hspace{0.025\textwidth}
  \begin{minipage}{0.40\textwidth}
    \begin{tikzpicture}[scale=0.63]
      \begin{loglogaxis}
        \addplot[style=solid,color=blue,mark=*] table[x=n_DOFs,y=errL2_u] {dat/k1_torus_tet.dat};
        \addplot[style=solid,color=red,mark=square*] table[x=n_DOFs,y=errL2_u] {dat/k2_torus_tet.dat};
        \addplot[style=solid,color=brown,mark=diamond*] table[x=n_DOFs,y=errL2_u] {dat/k3_torus_tet.dat};
        \addplot[style=dashed,color=blue,mark=*,mark options={solid}] table[x=n_DOFs,y=errL2_u] {dat/k1_torus_pol.dat};
        \addplot[style=dashed,color=red,mark=square*,mark options={solid}] table[x=n_DOFs,y=errL2_u] {dat/k2_torus_pol.dat};
        \addplot[style=dashed,color=brown,mark=diamond*,mark options={solid}] table[x=n_DOFs,y=errL2_u] {dat/k3_torus_pol.dat};
        \logLogSlopeTriangleNDOFs{0.10}{-0.4}{0.1}{2/3}{black};
        \logLogSlopeTriangleNDOFs{0.10}{-0.4}{0.1}{1}{black};
        \logLogSlopeTriangleNDOFs{0.10}{-0.4}{0.1}{4/3}{black};
      \end{loglogaxis}
    \end{tikzpicture}
  \end{minipage}
  \caption{\label{fig:torus}
    {\bf Test-case~\ref{sse:torus} (torus).} Relative energy-error ({\it top row}) and $L^2$-error ({\it bottom row}) vs.~$h_\d$ ({\it left column}) and \#DoF ({\it right column}) on tetrahedral ({\it solid}) and polyhedral ({\it dashed}) meshes.}
\end{figure}
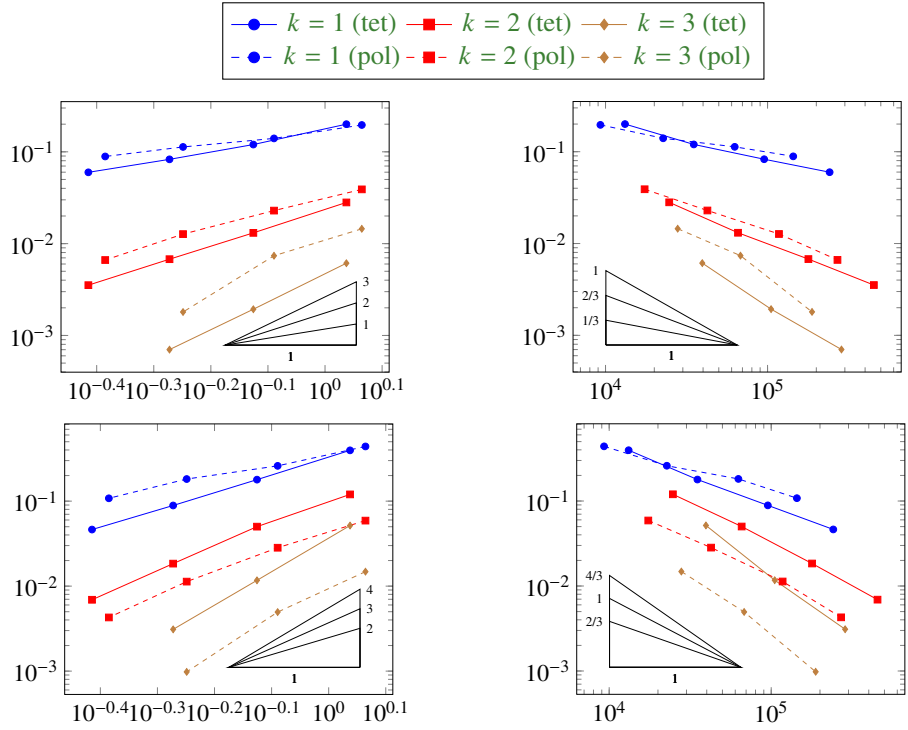

\subsection{Hollow ball} \label{sse:hollow.ball}

Our second example consists in a ball $\Omega$ of radius $2$, encapsulating a concentric hollow spheric cavity of radius $1$, so that its boundary $\Gamma$ consists of two connected components $\Gamma_0$ and $\Gamma_1$. For such a geometry, $\beta_1=0$, and $\beta_2 = 1$.
We consider the second-order Problem~\eqref{eq:vecpot.str}, with following smooth exact solution, expressed in spherical coordinates:
\ifSISC
\begin{equation} \label{eq:sol.hollow.ball}
  \vec{a}(r,\theta,\varphi)\defi\cos(\theta)/r^2\,\hat{\vec{r}}.
\end{equation}
\else
\begin{equation} \label{eq:sol.hollow.ball}
  \vec{a}(r,\theta,\varphi)\defi\frac{1}{r^2}\cos(\theta)\,\hat{\vec{r}}.
\end{equation}
\fi
Remark that $\vec{a}$ satisfies $\langle\vec{a}_{\mid\Gamma_1}{\cdot}\normal,1\rangle_{\Gamma_1}=0$ (as well as $\langle\vec{a}_{\mid\Gamma_0}{\cdot}\normal,1\rangle_{\Gamma_0}=0$, since $\Div\vec{a}=0$ in $\Omega$).
In our numerical experiments, the electric current density $\vec{j}$ is set according to~\eqref{eq:vecpot.str.a}, whereas the zero tangential boundary condition~\eqref{eq:vecpot.str.c} and zero flux condition~\eqref{eq:vecpot.str.d} (valid on the curved boundaries $\Gamma$ and $\Gamma_1$) are replaced by their non-homogeneous versions computed from~\eqref{eq:sol.hollow.ball} on the given piecewise polygonal approximations of $\Gamma$ and $\Gamma_1$.
For each $k\in\{1,2,3\}$, we solve Problem~\eqref{eq:vecpot.dis}, and plot on Figure~\ref{fig:hollow.ball} the relative energy-error and $L^2$-error as functions of the mesh size ($h_\d$) and of the number of degrees of freedom (\#DoF) after static condensation, for both the tetrahedral and polyhedral mesh families.
For both mesh families, we obtain, as predicted by Theorem~\ref{th:vecpot.esti}, a convergence rate for the energy-error of order $k$. We also observe a convergence rate of order $k+1$ for the $L^2$-error on the magnetic vector potential.

\begin{figure}[h!]
  \centering
  \ref{legend:hollow.ball}
  \vspace{0.25cm}\\
  \begin{minipage}{0.40\textwidth}
    \begin{tikzpicture}[scale=0.63]
      \begin{loglogaxis}[legend columns=3, legend to name=legend:hollow.ball]
        \addplot[style=solid,color=blue,mark=*] table[x=meshsize,y=errXnorm_u] {dat/k1_hollowball_tet.dat};
        \addplot[style=solid,color=red,mark=square*] table[x=meshsize,y=errXnorm_u] {dat/k2_hollowball_tet.dat};
        \addplot[style=solid,color=brown,mark=diamond*] table[x=meshsize,y=errXnorm_u] {dat/k3_hollowball_tet.dat};
        \addplot[style=dashed,color=blue,mark=*,mark options={solid}] table[x=meshsize,y=errXnorm_u] {dat/k1_hollowball_pol.dat};
        \addplot[style=dashed,color=red,mark=square*,mark options={solid}] table[x=meshsize,y=errXnorm_u] {dat/k2_hollowball_pol.dat};
        \addplot[style=dashed,color=brown,mark=diamond*,mark options={solid}] table[x=meshsize,y=errXnorm_u] {dat/k3_hollowball_pol.dat};
        \logLogSlopeTriangle{0.90}{0.4}{0.1}{1}{black};
        \logLogSlopeTriangle{0.90}{0.4}{0.1}{2}{black};
        \logLogSlopeTriangle{0.90}{0.4}{0.1}{3}{black};
        \legend{{\small $k=1$ (tet)},{\small $k=2$ (tet)},{\small $k=3$ (tet)},{\small $k=1$ (pol)},{\small $k=2$ (pol)},{\small $k=3$ (pol)}};
      \end{loglogaxis}
    \end{tikzpicture}
  \end{minipage}
  \hspace{0.025\textwidth}
  \begin{minipage}{0.40\textwidth}
    \begin{tikzpicture}[scale=0.63]
      \begin{loglogaxis}
        \addplot[style=solid,color=blue,mark=*] table[x=n_DOFs,y=errXnorm_u] {dat/k1_hollowball_tet.dat};
        \addplot[style=solid,color=red,mark=square*] table[x=n_DOFs,y=errXnorm_u] {dat/k2_hollowball_tet.dat};
        \addplot[style=solid,color=brown,mark=diamond*] table[x=n_DOFs,y=errXnorm_u] {dat/k3_hollowball_tet.dat};
        \addplot[style=dashed,color=blue,mark=*,mark options={solid}] table[x=n_DOFs,y=errXnorm_u] {dat/k1_hollowball_pol.dat};
        \addplot[style=dashed,color=red,mark=square*,mark options={solid}] table[x=n_DOFs,y=errXnorm_u] {dat/k2_hollowball_pol.dat};
        \addplot[style=dashed,color=brown,mark=diamond*,mark options={solid}] table[x=n_DOFs,y=errXnorm_u] {dat/k3_hollowball_pol.dat};
        \logLogSlopeTriangleNDOFs{0.10}{-0.4}{0.1}{1/3}{black};
        \logLogSlopeTriangleNDOFs{0.10}{-0.4}{0.1}{2/3}{black};
        \logLogSlopeTriangleNDOFs{0.10}{-0.4}{0.1}{1}{black};
      \end{loglogaxis}
    \end{tikzpicture}
  \end{minipage}
  \vspace{0.2cm}\\
  \begin{minipage}{0.40\textwidth}
    \begin{tikzpicture}[scale=0.63]
      \begin{loglogaxis}
        \addplot[style=solid,color=blue,mark=*] table[x=meshsize,y=errL2_u] {dat/k1_hollowball_tet.dat};
        \addplot[style=solid,color=red,mark=square*] table[x=meshsize,y=errL2_u] {dat/k2_hollowball_tet.dat};
        \addplot[style=solid,color=brown,mark=diamond*] table[x=meshsize,y=errL2_u] {dat/k3_hollowball_tet.dat};
        \addplot[style=dashed,color=blue,mark=*,mark options={solid}] table[x=meshsize,y=errL2_u] {dat/k1_hollowball_pol.dat};
        \addplot[style=dashed,color=red,mark=square*,mark options={solid}] table[x=meshsize,y=errL2_u] {dat/k2_hollowball_pol.dat};
        \addplot[style=dashed,color=brown,mark=diamond*,mark options={solid}] table[x=meshsize,y=errL2_u] {dat/k3_hollowball_pol.dat};
        \logLogSlopeTriangle{0.90}{0.4}{0.1}{2}{black};
        \logLogSlopeTriangle{0.90}{0.4}{0.1}{3}{black};
        \logLogSlopeTriangle{0.90}{0.4}{0.1}{4}{black};
      \end{loglogaxis}
    \end{tikzpicture}
  \end{minipage}
  \hspace{0.025\textwidth}
  \begin{minipage}{0.40\textwidth}
    \begin{tikzpicture}[scale=0.63]
      \begin{loglogaxis}
        \addplot[style=solid,color=blue,mark=*] table[x=n_DOFs,y=errL2_u] {dat/k1_hollowball_tet.dat};
        \addplot[style=solid,color=red,mark=square*] table[x=n_DOFs,y=errL2_u] {dat/k2_hollowball_tet.dat};
        \addplot[style=solid,color=brown,mark=diamond*] table[x=n_DOFs,y=errL2_u] {dat/k3_hollowball_tet.dat};
        \addplot[style=dashed,color=blue,mark=*,mark options={solid}] table[x=n_DOFs,y=errL2_u] {dat/k1_hollowball_pol.dat};
        \addplot[style=dashed,color=red,mark=square*,mark options={solid}] table[x=n_DOFs,y=errL2_u] {dat/k2_hollowball_pol.dat};
        \addplot[style=dashed,color=brown,mark=diamond*,mark options={solid}] table[x=n_DOFs,y=errL2_u] {dat/k3_hollowball_pol.dat};
        \logLogSlopeTriangleNDOFs{0.10}{-0.4}{0.1}{2/3}{black};
        \logLogSlopeTriangleNDOFs{0.10}{-0.4}{0.1}{1}{black};
        \logLogSlopeTriangleNDOFs{0.10}{-0.4}{0.1}{4/3}{black};
      \end{loglogaxis}
    \end{tikzpicture}
  \end{minipage}
  \caption{\label{fig:hollow.ball}
    {\bf Test-case~\ref{sse:hollow.ball} (hollow ball).} Relative energy-error ({\it top row}) and $L^2$-error ({\it bottom row}) vs.~$h_\d$ ({\it left column}) and \#DoF ({\it right column}) on tetrahedral ({\it solid}) and polyhedral ({\it dashed}) meshes.}
\end{figure}
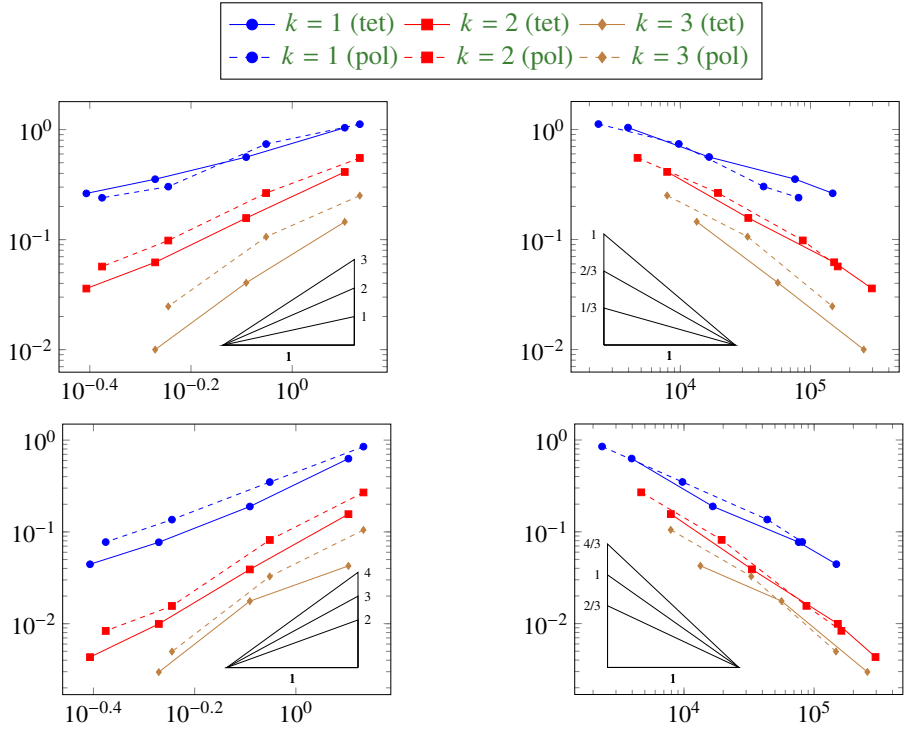

\begin{figure}[h!]
  \begin{center}
    \includegraphics[scale=0.25]{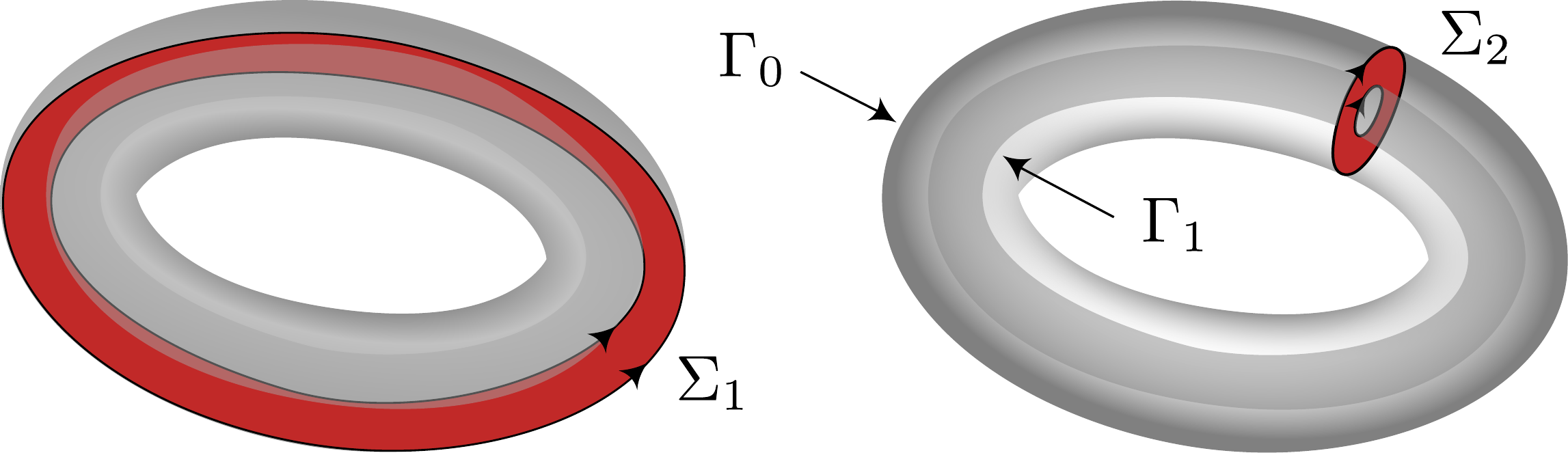}
  \end{center}
  \caption{\label{fig:hollow.torus.geo}
    {\bf Test-case~\ref{sse:hollow.torus} (toroidal vacuum chamber).} The toroidal domain $\Omega$, with coaxial hollow toric cavity. Depicted are two possible cutting surfaces $\Sigma_1$ and $\Sigma_2$, as well as the two connected components $\Gamma_0$ and $\Gamma_1$ of $\Gamma$.}
\end{figure}

\subsection{Toroidal vacuum chamber} \label{sse:hollow.torus}

This test-case is defined on the same toroidal domain as in Section~\ref{sse:torus}, except that $\Omega$ here additionally encapsulates a coaxial (i.e.~with midline radius $R$) hollow toric cavity of internal radius $r/2$ (that is, the cavity has half the internal radius of the torus).
For this geometry, $\beta_1=2$, and $\beta_2=1$.
We solve here the field formulation, with the same exact solution~\eqref{eq:sol.torus} (also same value for the constant $c$) as in the fully solid torus test-case of Section~\ref{sse:torus}.
We depict on Figure~\ref{fig:hollow.torus.geo} two possible cutting surfaces $\Sigma_1$ and $\Sigma_2$, as well as the two connected components $\Gamma_0$ and $\Gamma_1$ of the boundary $\Gamma$.
The vector field $\vec{h}$ given by~\eqref{eq:sol.torus} is actually solution to a non-homogeneous version of the first-order Problem~\eqref{eq:field.str}: it satisfies~\eqref{eq:field.str.b}--\eqref{eq:field.str.c} and $\langle\vec{h}_{\mid\Sigma_1}{\cdot}\normal_{\Sigma_1},1\rangle_{\Sigma_1}=0$, but $\langle\vec{h}_{\mid\Sigma_2}{\cdot}\normal_{\Sigma_2},1\rangle_{\Sigma_2}\neq 0$, so that $\vec{h}$ ``embeds'' an harmonic field in $\sharmonic$ (cf.~\eqref{eq:helm2}).
Here again, for $k\in\{1,2,3\}$, we solve Problem~\eqref{eq:field.dis} (with $\Curl_\t$), and plot on Figure~\ref{fig:hollow.torus.err} the relative energy-error and $L^2$-error as functions of $h_\d$ and \#DoF, for both the tetrahedral and polyhedral mesh families.
The obtained energy-error convergence rates are again consistent with our theory, as established in Theorem~\ref{th:field.esti}.
The observed $L^2$-error convergence rates on the magnetic field are again of order $k+1$.

\begin{figure}[h!]
  \centering
  \ref{legend:hollow.torus}
  \vspace{0.25cm}\\
  \begin{minipage}{0.40\textwidth}
    \begin{tikzpicture}[scale=0.63]
      \begin{loglogaxis}[legend columns=3, legend to name=legend:hollow.torus]
        \addplot[style=solid,color=blue,mark=*] table[x=meshsize,y=errXnorm_u] {dat/k1_hollowtorus_tet.dat};
        \addplot[style=solid,color=red,mark=square*] table[x=meshsize,y=errXnorm_u] {dat/k2_hollowtorus_tet.dat};
        \addplot[style=solid,color=brown,mark=diamond*] table[x=meshsize,y=errXnorm_u] {dat/k3_hollowtorus_tet.dat};
        \addplot[style=dashed,color=blue,mark=*,mark options={solid}] table[x=meshsize,y=errXnorm_u] {dat/k1_hollowtorus_pol.dat};
        \addplot[style=dashed,color=red,mark=square*,mark options={solid}] table[x=meshsize,y=errXnorm_u] {dat/k2_hollowtorus_pol.dat};
        \addplot[style=dashed,color=brown,mark=diamond*,mark options={solid}] table[x=meshsize,y=errXnorm_u] {dat/k3_hollowtorus_pol.dat};
        \logLogSlopeTriangle{0.90}{0.4}{0.05}{1}{black};
        \logLogSlopeTriangle{0.90}{0.4}{0.05}{2}{black};
        \logLogSlopeTriangle{0.90}{0.4}{0.05}{3}{black};
        \legend{{\small $k=1$ (tet)},{\small $k=2$ (tet)},{\small $k=3$ (tet)},{\small $k=1$ (pol)},{\small $k=2$ (pol)},{\small $k=3$ (pol)}};          
      \end{loglogaxis}
    \end{tikzpicture}
  \end{minipage}
  \hspace{0.025\textwidth}
  \begin{minipage}{0.40\textwidth}
    \begin{tikzpicture}[scale=0.63]
      \begin{loglogaxis}
        \addplot[style=solid,color=blue,mark=*] table[x=n_DOFs,y=errXnorm_u] {dat/k1_hollowtorus_tet.dat};
        \addplot[style=solid,color=red,mark=square*] table[x=n_DOFs,y=errXnorm_u] {dat/k2_hollowtorus_tet.dat};
        \addplot[style=solid,color=brown,mark=diamond*] table[x=n_DOFs,y=errXnorm_u] {dat/k3_hollowtorus_tet.dat};
        \addplot[style=dashed,color=blue,mark=*,mark options={solid}] table[x=n_DOFs,y=errXnorm_u] {dat/k1_hollowtorus_pol.dat};
        \addplot[style=dashed,color=red,mark=square*,mark options={solid}] table[x=n_DOFs,y=errXnorm_u] {dat/k2_hollowtorus_pol.dat};
        \addplot[style=dashed,color=brown,mark=diamond*,mark options={solid}] table[x=n_DOFs,y=errXnorm_u] {dat/k3_hollowtorus_pol.dat};
        \logLogSlopeTriangleNDOFs{0.10}{-0.4}{0.1}{1/3}{black};
        \logLogSlopeTriangleNDOFs{0.10}{-0.4}{0.1}{2/3}{black};
        \logLogSlopeTriangleNDOFs{0.10}{-0.4}{0.1}{1}{black};
      \end{loglogaxis}
    \end{tikzpicture}
  \end{minipage}
  \vspace{0.2cm}\\
  \begin{minipage}{0.40\textwidth}
    \begin{tikzpicture}[scale=0.63]
      \begin{loglogaxis}
        \addplot[style=solid,color=blue,mark=*] table[x=meshsize,y=errL2_u] {dat/k1_hollowtorus_tet.dat};
        \addplot[style=solid,color=red,mark=square*] table[x=meshsize,y=errL2_u] {dat/k2_hollowtorus_tet.dat};
        \addplot[style=solid,color=brown,mark=diamond*] table[x=meshsize,y=errL2_u] {dat/k3_hollowtorus_tet.dat};
        \addplot[style=dashed,color=blue,mark=*,mark options={solid}] table[x=meshsize,y=errL2_u] {dat/k1_hollowtorus_pol.dat};
        \addplot[style=dashed,color=red,mark=square*,mark options={solid}] table[x=meshsize,y=errL2_u] {dat/k2_hollowtorus_pol.dat};
        \addplot[style=dashed,color=brown,mark=diamond*,mark options={solid}] table[x=meshsize,y=errL2_u] {dat/k3_hollowtorus_pol.dat};
        \logLogSlopeTriangle{0.90}{0.4}{0.05}{2}{black};
        \logLogSlopeTriangle{0.90}{0.4}{0.05}{3}{black};
        \logLogSlopeTriangle{0.90}{0.4}{0.05}{4}{black};
      \end{loglogaxis}
    \end{tikzpicture}
  \end{minipage}
  \hspace{0.025\textwidth}
  \begin{minipage}{0.40\textwidth}
    \begin{tikzpicture}[scale=0.63]
      \begin{loglogaxis}
        \addplot[style=solid,color=blue,mark=*] table[x=n_DOFs,y=errL2_u] {dat/k1_hollowtorus_tet.dat};
        \addplot[style=solid,color=red,mark=square*] table[x=n_DOFs,y=errL2_u] {dat/k2_hollowtorus_tet.dat};
        \addplot[style=solid,color=brown,mark=diamond*] table[x=n_DOFs,y=errL2_u] {dat/k3_hollowtorus_tet.dat};
        \addplot[style=dashed,color=blue,mark=*,mark options={solid}] table[x=n_DOFs,y=errL2_u] {dat/k1_hollowtorus_pol.dat};
        \addplot[style=dashed,color=red,mark=square*,mark options={solid}] table[x=n_DOFs,y=errL2_u] {dat/k2_hollowtorus_pol.dat};
        \addplot[style=dashed,color=brown,mark=diamond*,mark options={solid}] table[x=n_DOFs,y=errL2_u] {dat/k3_hollowtorus_pol.dat};
        \logLogSlopeTriangleNDOFs{0.10}{-0.4}{0.1}{2/3}{black};
        \logLogSlopeTriangleNDOFs{0.10}{-0.4}{0.1}{1}{black};
        \logLogSlopeTriangleNDOFs{0.10}{-0.4}{0.1}{4/3}{black};
      \end{loglogaxis}
    \end{tikzpicture}
  \end{minipage}
  \caption{\label{fig:hollow.torus.err}
    {\bf Test-case~\ref{sse:hollow.torus} (toroidal vacuum chamber).} Relative energy-error ({\it top row}) and $L^2$-error ({\it bottom row}) vs.~$h_\d$ ({\it left column}) and \#DoF ({\it right column}) on tetrahedral ({\it solid}) and polyhedral ({\it dashed}) meshes.}
\end{figure}
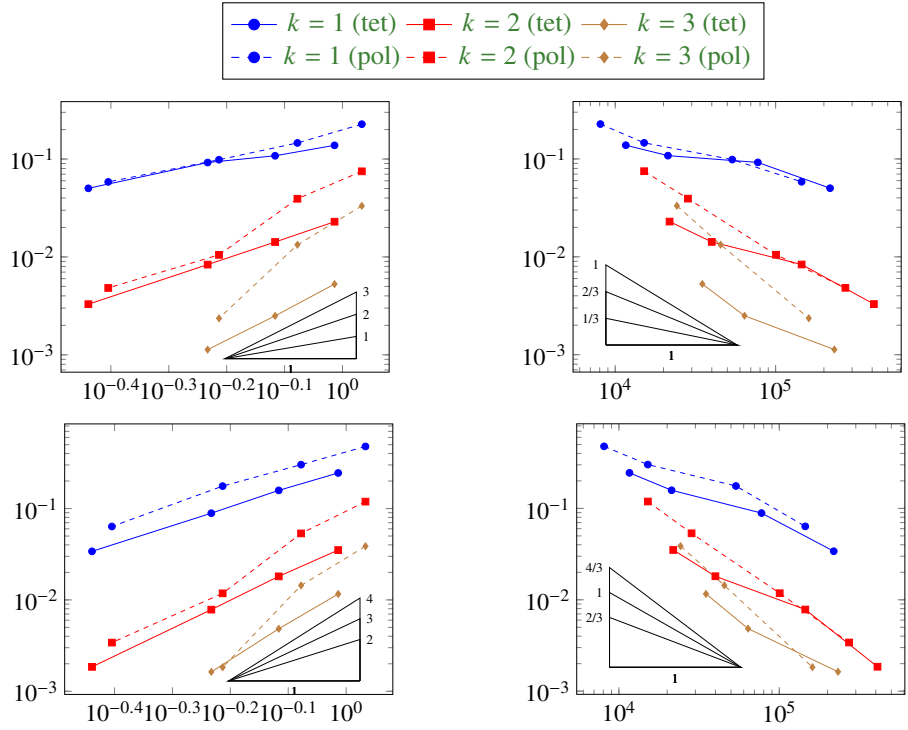

\subsection{Variable permeability} \label{sse:variable}

In this example, we let $\Omega\defi(0,1)^3$, and we solve the vector potential formulation, with smooth exact solution
\begin{equation} \label{eq:sol.variable}
  \vec{a}(x,y,z) \defi \left(
    \begin{tabular}{l}
      $2\pi r(x) \sin(2\pi y) \cos(2 \pi z)$
      \\
      $- r'(x) \cos(2 \pi y)\cos(2 \pi z)$
      \\
      $-2 r'(x) \sin(2 \pi y)\sin(2 \pi z)$
    \end{tabular}
  \right),\qquad r(x)\defi x^3,
\end{equation}
together with smooth, locally variable (inverse) magnetic permeability $\mu^{-1}(x,y,z)\defi 1 + x^2 y^2$.
The vector field $\vec{a}$ given by~\eqref{eq:sol.variable} is solution to a non-homogeneous version of the second-order Problem~\eqref{eq:vecpot.str}: it satisfies~\eqref{eq:vecpot.str.b} (note that~\eqref{eq:vecpot.str.d} is here trivial), but fulfills in place of~\eqref{eq:vecpot.str.c} a non-homogeneous tangential boundary condition.
In our experiments, the electric current density $\vec{j}$ is set according to~\eqref{eq:vecpot.str.a}, and the boundary datum in accordance with~\eqref{eq:sol.variable}.
Remark that $\mu^{-1}$ is not piecewise constant over the partition of the domain. So as to retain optimal convergence rates, we thus had in our implementation to use weighted cell/face inner products (remark that $\mu^{-1}$ is continuous across interfaces). The relative energy-error is computed accordingly. Their accurate computation necessitates to increase the degree of exactness of the quadrature rules.
We here solve Problem~\eqref{eq:vecpot.dis}. The convergence results, collected in Figure~\ref{fig:variable}, show essentially the same behavior as those reported in~\cite[Sec.~3.2.5]{CDPLe:22} for a unit permeability test-case on the same domain.

\begin{figure}[h!]
  \centering
  \ref{legend:variable}
  \vspace{0.25cm}\\
  \begin{minipage}{0.40\textwidth}
    \begin{tikzpicture}[scale=0.63]
      \begin{loglogaxis}[legend columns=3, legend to name=legend:variable]
        \addplot[style=solid,color=blue,mark=*] table[x=meshsize,y=errXnorm_u] {dat/k1_cube_tet.dat};
        \addplot[style=solid,color=red,mark=square*] table[x=meshsize,y=errXnorm_u] {dat/k2_cube_tet.dat};
        \addplot[style=solid,color=brown,mark=diamond*] table[x=meshsize,y=errXnorm_u] {dat/k3_cube_tet.dat};
        \addplot[style=dashed,color=blue,mark=*,mark options={solid}] table[x=meshsize,y=errXnorm_u] {dat/k1_cube_pol.dat};
        \addplot[style=dashed,color=red,mark=square*,mark options={solid}] table[x=meshsize,y=errXnorm_u] {dat/k2_cube_pol.dat};
        \addplot[style=dashed,color=brown,mark=diamond*,mark options={solid}] table[x=meshsize,y=errXnorm_u] {dat/k3_cube_pol.dat};
        \logLogSlopeTriangle{0.90}{0.4}{0.1}{1}{black};
        \logLogSlopeTriangle{0.90}{0.4}{0.1}{2}{black};
        \logLogSlopeTriangle{0.90}{0.4}{0.1}{3}{black};
        \legend{{\small $k=1$ (tet)},{\small $k=2$ (tet)},{\small $k=3$ (tet)},{\small $k=1$ (pol)},{\small $k=2$ (pol)},{\small $k=3$ (pol)}};          
      \end{loglogaxis}
    \end{tikzpicture}
  \end{minipage}
  \hspace{0.025\textwidth}
  \begin{minipage}{0.40\textwidth}
    \begin{tikzpicture}[scale=0.63]
      \begin{loglogaxis}
        \addplot[style=solid,color=blue,mark=*] table[x=n_DOFs,y=errXnorm_u] {dat/k1_cube_tet.dat};
        \addplot[style=solid,color=red,mark=square*] table[x=n_DOFs,y=errXnorm_u] {dat/k2_cube_tet.dat};
        \addplot[style=solid,color=brown,mark=diamond*] table[x=n_DOFs,y=errXnorm_u] {dat/k3_cube_tet.dat};
        \addplot[style=dashed,color=blue,mark=*,mark options={solid}] table[x=n_DOFs,y=errXnorm_u] {dat/k1_cube_pol.dat};
        \addplot[style=dashed,color=red,mark=square*,mark options={solid}] table[x=n_DOFs,y=errXnorm_u] {dat/k2_cube_pol.dat};
        \addplot[style=dashed,color=brown,mark=diamond*,mark options={solid}] table[x=n_DOFs,y=errXnorm_u] {dat/k3_cube_pol.dat};
        \logLogSlopeTriangleNDOFs{0.10}{-0.4}{0.1}{1/3}{black};
        \logLogSlopeTriangleNDOFs{0.10}{-0.4}{0.1}{2/3}{black};
        \logLogSlopeTriangleNDOFs{0.10}{-0.4}{0.1}{1}{black};
      \end{loglogaxis}
    \end{tikzpicture}
  \end{minipage}
  \vspace{0.2cm}\\
  \begin{minipage}{0.40\textwidth}
    \begin{tikzpicture}[scale=0.63]
      \begin{loglogaxis}
        \addplot[style=solid,color=blue,mark=*] table[x=meshsize,y=errL2_u] {dat/k1_cube_tet.dat};
        \addplot[style=solid,color=red,mark=square*] table[x=meshsize,y=errL2_u] {dat/k2_cube_tet.dat};
        \addplot[style=solid,color=brown,mark=diamond*] table[x=meshsize,y=errL2_u] {dat/k3_cube_tet.dat};
        \addplot[style=dashed,color=blue,mark=*,mark options={solid}] table[x=meshsize,y=errL2_u] {dat/k1_cube_pol.dat};
        \addplot[style=dashed,color=red,mark=square*,mark options={solid}] table[x=meshsize,y=errL2_u] {dat/k2_cube_pol.dat};
        \addplot[style=dashed,color=brown,mark=diamond*,mark options={solid}] table[x=meshsize,y=errL2_u] {dat/k3_cube_pol.dat};
        \logLogSlopeTriangle{0.90}{0.4}{0.1}{2}{black};
        \logLogSlopeTriangle{0.90}{0.4}{0.1}{3}{black};
        \logLogSlopeTriangle{0.90}{0.4}{0.1}{4}{black};
      \end{loglogaxis}
    \end{tikzpicture}
  \end{minipage}
  \hspace{0.025\textwidth}
  \begin{minipage}{0.40\textwidth}
    \begin{tikzpicture}[scale=0.63]
      \begin{loglogaxis}
        \addplot[style=solid,color=blue,mark=*] table[x=n_DOFs,y=errL2_u] {dat/k1_cube_tet.dat};
        \addplot[style=solid,color=red,mark=square*] table[x=n_DOFs,y=errL2_u] {dat/k2_cube_tet.dat};
        \addplot[style=solid,color=brown,mark=diamond*] table[x=n_DOFs,y=errL2_u] {dat/k3_cube_tet.dat};
        \addplot[style=dashed,color=blue,mark=*,mark options={solid}] table[x=n_DOFs,y=errL2_u] {dat/k1_cube_pol.dat};
        \addplot[style=dashed,color=red,mark=square*,mark options={solid}] table[x=n_DOFs,y=errL2_u] {dat/k2_cube_pol.dat};
        \addplot[style=dashed,color=brown,mark=diamond*,mark options={solid}] table[x=n_DOFs,y=errL2_u] {dat/k3_cube_pol.dat};
        \logLogSlopeTriangleNDOFs{0.10}{-0.4}{0.1}{2/3}{black};
        \logLogSlopeTriangleNDOFs{0.10}{-0.4}{0.1}{1}{black};
        \logLogSlopeTriangleNDOFs{0.10}{-0.4}{0.1}{4/3}{black};
      \end{loglogaxis}
    \end{tikzpicture}
  \end{minipage}
  \caption{\label{fig:variable}
    {\bf Test-case~\ref{sse:variable} (variable permeability).} Relative energy-error ({\it top row}) and $L^2$-error ({\it bottom row}) versus $h_\d$ ({\it left column}) and \#DoF ({\it right column}) on tetrahedral ({\it solid}) and polyhedral ({\it dashed}) meshes.}
\end{figure}
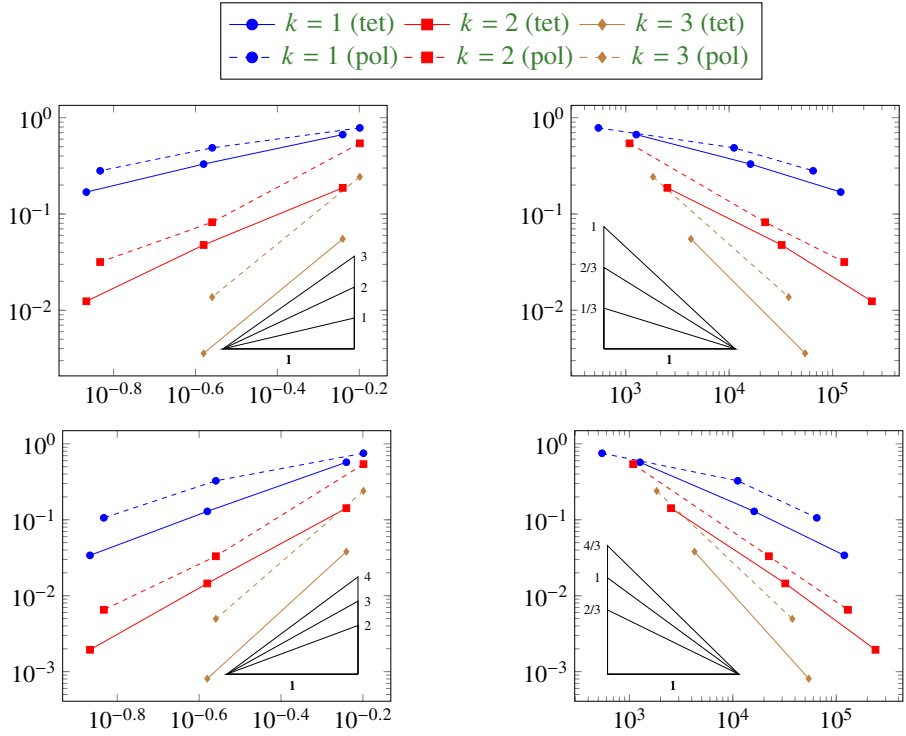

\subsection{Singular solution} \label{sse:singular}

In this last test-case, the domain $\Omega$ is a cylinder of unit radius and height, to which a quarter has been removed, so that it presents a reentrant edge (cf.~the left panel of Figure~\ref{fig:singular}).
We here solve the field formulation, for the following singular exact solution, expressed in cylindrical coordinates:
\ifSISC
\begin{equation} \label{eq:sol.singular}
  \vec{h}(\rho,\varphi,z)\defi\Grad\big(\rho^{\frac{2}{3}}\cos(2\varphi/3)\big).
\end{equation}
\else
\begin{equation} \label{eq:sol.singular}
  \vec{h}(\rho,\varphi,z)\defi\Grad\left(\rho^{\frac{2}{3}}\cos(\frac{2}{3}\varphi)\right).
\end{equation}
\fi
The vector field $\vec{h}$ given by~\eqref{eq:sol.singular} is solution to a non-homogeneous version of the first-order Problem~\eqref{eq:field.str}: it does satisfy~\eqref{eq:field.str.b} (remark that~\eqref{eq:field.str.d} is here trivial), but fulfills in place of~\eqref{eq:field.str.c} a non-homogeneous normal boundary condition on the curved section of $\Gamma$.
In our numerical experiments, the electric current density $\vec{j}$ is set to zero (accordingly to~\eqref{eq:field.str.a}), and the boundary datum is set in accordance with~\eqref{eq:sol.singular} on the given piecewise polygonal approximation of the curved section of $\Gamma$.
It can be verified that $\vec{h}\in\vec{H}^{\frac{2}{3}-\varepsilon}(\Omega)$ for all $\varepsilon>0$, with $\vec{h}$ being singular along the reentrant edge.
We solve Problem~\eqref{eq:field.dis} (with $\CTd$) for $k=1$ on the tetrahedral mesh sequence.
Note that the regularity assumptions of Theorem~\ref{th:field.esti} are violated, hence our theory does not apply here.
We collect on the right panel of Figure~\ref{fig:singular} the results for the relative $L^2$-error on the magnetic field.
We observe that the $L^2$-error attains its maximum possible rate of convergence of $2/3$; no convergence is observed in energy-norm (not reported).
We verified, for each $k\in\{2,3\}$, that the $L^2$-error convergence rate does not depend on the polynomial degree.
Finally, we performed the exact same test, but with discrete rotational operator given by $\Curl_\t$ instead of $\CTd$. We observed no convergence whatsoever, even in $L^2$-norm. This last observation corroborates the discussion in Remark~\ref{re:variant}.

\begin{figure}
  \begin{minipage}{0.50\textwidth}
    \hspace*{1cm}
    \includegraphics[scale=0.14]{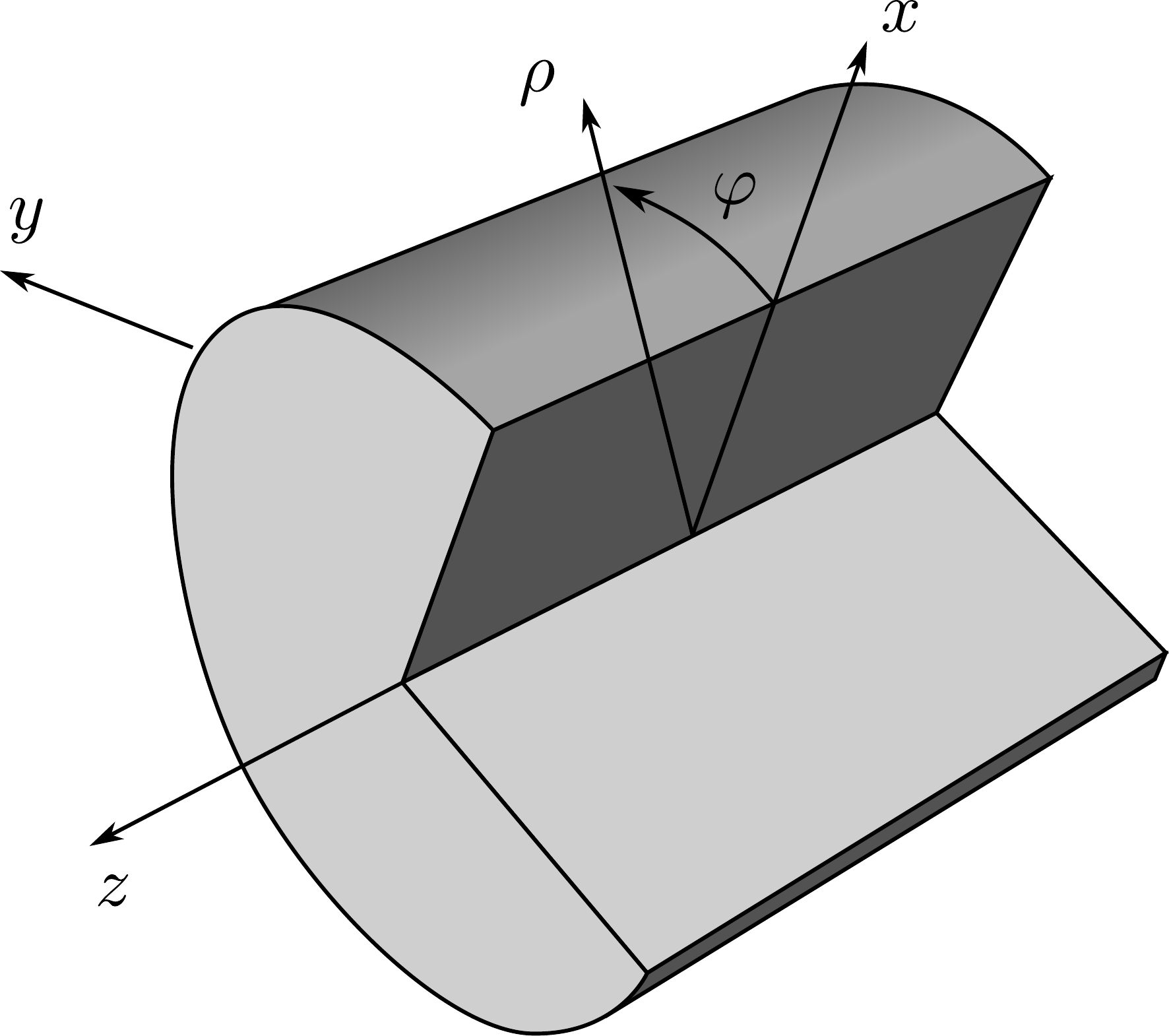}
  \end{minipage}
  \begin{minipage}{0.50\textwidth}
    \vspace*{0.4cm}
    {\small
    \begin{tabular}{c|c|c}
      \hline
      $h_\d$ & $L^2$-error & Rate \\
      \hline\hline
      $4.20\mathrm{E}{-01}$  & $3.62\mathrm{E}{-02}$ & $-$\\
      \hline
      $2.37\mathrm{E}{-01}$  & $2.72\mathrm{E}{-02}$ & $0.50$\\
      \hline
      $1.44\mathrm{E}{-01}$  & $1.99\mathrm{E}{-02}$ & $0.63$\\
      \hline
      $7.00\mathrm{E}{-02}$  & $1.24\mathrm{E}{-02}$ & $0.66$\\
      \hline
    \end{tabular}
    }
  \end{minipage}
  \caption{\label{fig:singular}
    {\bf Test-case~\ref{sse:singular} (singular solution).} The computational domain $\Omega$ ({\it left panel}). Relative $L^2$-error and convergence rate on tetrahedral meshes for $k=1$ ({\it right panel}).}
\end{figure}

\ifSISC
\section*{Acknowledgments}
{\footnotesize The work of the authors is supported by the action ``Pr\'eservation de l'emploi de R\&D'' from the ``Plan de Relance'' (recovery plan) of the French State, as well as by the Agence Nationale de la Recherche (ANR) under the PRCE grant HIPOTHEC (ANR-23-CE46-0013).
The last two authors also acknowledge support from the LabEx CEMPI (ANR-11-LABX-0007).}
\else
\section*{Acknowledgments}
The work of the authors is supported by the action ``Pr\'eservation de l'emploi de R\&D'' from the ``Plan de Relance'' (recovery plan) of the French State, as well as by the Agence Nationale de la Recherche (ANR) under the PRCE grant HIPOTHEC (ANR-23-CE46-0013).
The last two authors also acknowledge support from the LabEx CEMPI (ANR-11-LABX-0007).
\fi

\ifSISC
\bibliographystyle{siamplain}
\bibliography{hybrid_divcurl}
\else
\bibliographystyle{plain}
{\small
\bibliography{hybrid_divcurl}
}
\fi

\end{document}